\documentclass[11pt]{amsart}
\usepackage[hmargin=2.5cm,vmargin=2.5cm]{geometry}

\usepackage{amssymb}
\usepackage{amsmath}
\usepackage{amsfonts}
\usepackage{amsthm}
\usepackage{stmaryrd}
\usepackage[all]{xy}
\usepackage{mathrsfs}
\usepackage{graphicx}
\usepackage{hyperref}
\usepackage{color}
\usepackage{multirow}
\usepackage{extarrows}
\usepackage{amscd}
\usepackage{scalerel}
\usepackage{stackengine}
\usepackage{bbm}
\usepackage{mathtools}
\usepackage{upgreek}

\usepackage{mathdots}

\usepackage{cite}
\usepackage[all]{hypcap}        
\usepackage{pictexwd,dcpic}

\numberwithin{equation}{subsubsection}
\newtheorem{theorem}[subsubsection]{Theorem}

\newtheorem{corollary}[subsubsection]{Corollary}
\newtheorem{lemma}[subsubsection]{Lemma}
\newtheorem{proposition}[subsubsection]{Proposition}

\theoremstyle{definition}

\newtheorem{remark}[subsubsection]{Remark}

\newcommand{\ra}{\rightarrow}
\newcommand{\lra}{\longrightarrow}

\def\AAA{\mathbb{A}}

\def\CC{\mathbb{C}}

\def\ZZ{\mathbb{Z}}

\def\calA{\mathcal{A}}

\def\calF{\mathcal{F}}

\def\calH{\mathcal{H}}

\def\calJ{\mathcal{J}}

\def\calL{\mathcal{L}}
\def\calM{\mathcal{M}}

\def\calR{\mathcal{R}}
\def\calS{\mathcal{S}}
\def\calT{\mathcal{T}}

\def\calV{\mathcal{V}}
\def\calW{\mathcal{W}}

\def\scrR{\mathscr{R}}
\def\scrS{\mathscr{S}}

\def\scrV{\mathscr{V}}
\def\scrW{\mathscr{W}}

\def\wt{\widetilde}
\def\wh{\widehat}

\def\alg_k{\AAA\mathbbm{l}\mathbbm{g}_{/k}}

\DeclareMathOperator{\End}{End}

\DeclareMathOperator{\Hom}{Hom}

\DeclareMathOperator{\id}{id}
\DeclareMathOperator{\Image}{Im}
\DeclareMathOperator{\Ind}{Ind}

\DeclareMathOperator{\Res}{Res}

\DeclareMathOperator{\Irr}{Irr}
\DeclareMathOperator{\Irrt}{Irr_{temp}}
\DeclareMathOperator{\Para}{\Phi}
\DeclareMathOperator{\Parat}{\Phi_{temp}}

\DeclareMathOperator{\Herm}{Herm}
\DeclareMathOperator{\Isom}{Isom}
\DeclareMathOperator{\Tr}{Tr}
\DeclareMathOperator{\disc}{disc}
\DeclareMathOperator{\Nm}{Nm}

\begin{document}

\title{Local Langlands Correspondence for Unitary Groups via Theta Lifts}
\author{Rui Chen \and Jialiang Zou}

\begin{abstract}
Using the theta correspondence, we extend the classification of irreducible representations of quasi-split unitary groups (the so-called local Langlands correspondence) due to \cite{MR3338302} to non quasi-split unitary groups. We also prove that our classification satisfies some good properties, which characterize it uniquely. In particular, this paper provides an alternative approach to the works of \cite{kaletha2014endoscopic} and \cite{MR3839702}.
\end{abstract}

\maketitle

\section{Introduction}
In his monumental book \cite{MR3135650}, Arthur gave a complete description of the automorphic discrete spectra of quasi-split orthogonal groups and symplectic groups, by using the stable trace formula and the theory of endoscopy. One of the main local theorems in that book is the local Langlands correspondence (``LLC'' for short), which gives a classification of irreducible tempered representations of quasi-split classical groups. Following Arthur's method, Mok established the same results for quasi-split unitary groups \cite{MR3338302}. To extend these results to non quasi-split classical groups, one can use the stable trace formula \`a la Arthur. This was partially carried out by Kaletha-M\'{\i}nguez-Shin-White in \cite{kaletha2014endoscopic} for unitary groups. In particular, they established the LLC for all unitary groups, in the enhanced version of Vogan. M{\oe}glin-Renard also have some related results \cite{MR3839702}. Both these two papers use very difficult techniques.\\

However, the theta correspondence provides us a rather cheap way to establish, or, ``transfer'' results from one group to another group. Indeed, this idea has been used in many papers, for example, \cite{MR2800725}, \cite{MR2999299}, \cite{MR3866889}, and a recent paper \cite{MR4157428}. This paper is another exploitation of this idea. The main goal of this paper is to construct a (Vogan version) LLC for unitary groups over a $p$-adic field, based on the LLC for quasi-split unitary groups. We will also prove that this LLC satisfies several desired properties; these properties will uniquely determine the LLC (see Theorem \ref{MainTheorem}). Among these properties, the most important one is so-called ``local intertwining relations'' (``LIR'' for short), which allows us to distinguish representations in a tempered $L$-packet by using (normalized) intertwining operators. We would like to remark here that the LIR we used here is the same as in \cite{MR3573972}, which is a little bit different from the LIR formulated by Arthur/ Mok/ KMSW (see Remark \ref{RMK.ExistedResults&LIRs}). As in other instances where the LLC was shown using the theta correspondence (such as \cite{MR2800725} and \cite{MR2999299}), we do not show the (twisted) endoscopic character identities for the $L$-packets we constructed. To show that our $L$-packets satisfy the endoscopic character identities, one would need to appeal to the stable trace formula (or a simple form of it), as was done in \cite{MR3267112} and \cite{luo2020endoscopic}. Although essentially there is no new result in this paper, it provides an alternative approach to the works of \cite{kaletha2014endoscopic} and \cite{MR3839702}.\\

We would like to mention some related works. In \cite{MR3573972}, Gan-Ichino proved the so-called Prasad conjecture, which describes the almost equal rank theta lifts in terms of the LLC; similarly, in \cite{MR3714507}, Atobe-Gan described the theta lifts of tempered representations in terms of the LLC. In this paper, we ``turn the table around'', namely, imitating the prediction of Prasad conjecture, we construct a Vogan version LLC for unitary groups. We also write a parallel paper \cite{MR4308934}, in which we use the same method to deal with the even orthogonal groups (we write it separately to avoid making notations too complicated). In a sequel to this paper, we carry out the global counterpart of this paper and establish the Arthur's multiplicity formula for the tempered part of automorphic discrete spectra of even orthogonal groups/ unitary groups.\\

We now give a summary of the layout of this paper. We formulate the main theorem (i.e. the desired LLC, Theorem \ref{MainTheorem}) in Section 2, taking the chance to recall some results from \cite{MR3338302} that we are using. After recalling some basics of theta correspondence in Section 3, we give our construction in Section 4, and prove several properties of the desired LLC along the way. Then in Section 5 we recall some results from \cite{MR3573972}, which will be the ingredients in the proof of the LIR in Section 6. Finally in Section 7, with the help of the LIR, we are able to finish the proof of the main theorem. To keep the paper in a reasonable length, we omit many repeated details. Readers can refer to the arXiv version of this paper if they would like the full details.

\section*{Acknowledgments}
We would like to thank our supervisor Wee Teck Gan for many useful advice. We also thank Hiraku Atobe, Atsushi Ichino, Wen-Wei Li, and Sug Woo Shin for helpful conversations during the conference ``Workshop on Shimura varieties, representation theory and related topics, 2019'' in Hokkaido University. We thank Hiroshi Ishimoto, Caihua Luo, Xiaolei Wan, and Chuijia Wang for helpful discussions. Both authors are supported by an MOE Graduate Research Scholarship.

\section*{Notation}
We set up some notations at the begining of this paper. Let $F$ be a non-Archimedean local field of characteristic $0$ and residue characteristic $p$. Let $E$ be a quadratic field extension of $F$ and let $\omega_{E/F}$ be the quadratic character of $F^\times$ associated to $E/F$ by local class field theory. We denote by $c$ the non-trivial Galois automorphism of $E$ over $F$. Let $\Tr_{E/F}$ and $\Nm_{E/F}$ be the trace and norm maps from $E$ to $F$. We denote by $E^1$ the subgroup of $E^\times$ consisting of norm $1$ elements. We choose an element $\delta\in E^\times$ such that $\Tr_{E/F}(\delta)=0$. We write $|\cdot|=|\cdot|_E$ for the normalized absolute value on $E$. If $\psi$ is an additive character of $F$, we shall use $\psi_E$ to denote the additive character of $E$ defined by $\psi_E=\psi\circ\Tr_{E/F}$. If $\pi$ is a representation of some group $G$, we shall use $\pi^\vee$ to denote the contragredient of $\pi$.

\section{Local Langlands Correspondence}
In this section, we formulate the desired LLC for unitary groups.

\subsection{Hermitian and skew-Hermitian spaces}
Fix $\varepsilon=\pm1$. Let $V$ be a finite dimensional vector space over $E$ equipped with a non-degenerate $\varepsilon$-Hermitian form
\begin{equation*}
\langle\cdot,\cdot\rangle_V: V\times V\lra E.
\end{equation*}
Put $n=\dim V$ and $\disc V=(-1)^{(n-1)n/2}\cdot\det V$, so that
\begin{equation*}
\disc V\in\begin{cases}F^\times/\Nm_{E/F}(E^\times)\quad &\textit{if }\varepsilon=+1;\\
\delta^n\cdot F^\times/\Nm_{E/F}(E^\times)\quad &\textit{if }\varepsilon=-1.
\end{cases}
\end{equation*}
We define $\epsilon(V)=\pm1$ by
\begin{equation*}
\epsilon (V)=\begin{cases}\omega_{E/F}(\disc V)\quad &\textit{if }\varepsilon=+1;\\
\omega_{E/F}(\delta^{-n}\cdot\disc V)\quad &\textit{if }\varepsilon=-1.
\end{cases}
\end{equation*}
Given a positive integer $n$, there are precisely two isometry classes of $n$-dimensional $\varepsilon$-Hermitian spaces $V$, which are distinguished from each other by their signs $\epsilon(V)$. Note that 
\begin{itemize}
\item $\epsilon(V)$ depends on the choice of $\delta$ if $\varepsilon=-1$ and $n$ is odd;
\item $V^+$ always has the maximal possible Witt index $[\dim V^+/2]$.
\end{itemize}
Let $U(V)$ be the unitary group of $V$. If $n=0$, we interpret $U(V)$ as the trivial group $\{1\}$.\\

Sometimes we also need to consider a tower of $\varepsilon$-Hermitian spaces. Let $V_{an}$ be an anisotropic space over $E$, and for $r\geq 0$, let
\[
	V_{an,r}=V_{an}\oplus\calH^r,
\]
where $\calH$ is the ($\varepsilon$-Hermitian) hyperbolic plane. Let $U(V_{an,r})$ be the unitary group associated to $V_{an,r}$. The collection
\[
	\{V_{an,r}~|~r\geq 0\}
\]
is called a Witt tower of spaces. We note that any given $\varepsilon$-Hermitian space $V$ is a member of a unique Witt tower of spaces $\calV$.

\subsection{Langlands parameters and component groups}
Let $W_E$ be the Weil group of $E$ and $WD_E=W_E\times SL_2(\CC)$ the Weil-Deligne group of $E$. Recall that an $L$-parameter for the unitary group $U(V)$ is an $n$-dimensional conjugate self-dual representation of $WD_E$ 
\[
	\phi:WD_E\lra GL_n(\CC)
\]
with sign $(-1)^{n-1}$. Let $\Para(n)$ be the set of equivalence classes of $L$-parameters for unitary groups of $n$ variables. Given $\phi\in\Para(n)$, we may decompose it into a direct sum
\begin{equation*}
\phi=\bigoplus_im_i\phi_i
\end{equation*}
with pairwise inequivalent irreducible representations $\phi_i$ of $WD_E$ and multiplicities $m_i$. We say that $\phi$ is square-integrable if it is multiplicity-free and $\phi_i$ is conjugate self-dual with sign $(-1)^{n-1}$ for all $i$, and we say that $\phi$ is tempered if the image of $W_E$ is bounded.\\

For an $L$-parameter $\phi$ for $U(V)$, we can define the component group $\calS_\phi$ associated to $\phi$ following \cite{MR3202556} Section 8. If we write $\phi=\oplus_im_i\phi_i$, then $\calS_\phi$ has an explicit description of the form
\begin{equation*}
\calS_\phi=\prod_j(\ZZ/2\ZZ)a_j
\end{equation*} 
with a canonical basis $\{a_j\}$, where the product ranges over all $j$ such that $\phi_j$ is conjugate self-dual with sign $(-1)^{n-1}$. For $a=a_{j_1}+\cdots+a_{j_r}\in\calS_\phi$, we put 
\[
	\phi^a=\phi_{j_1}\oplus\cdots\oplus\phi_{j_r}.
\]
We shall let $z_\phi$ denote the image of $-1\in GL_n(\CC)$ in $\calS_\phi$. More explicitly, we have
\begin{equation*}
z_\phi=(m_ja_j)\in\prod_j(\ZZ/2\ZZ)a_j.
\end{equation*} 
Let $\overline{\calS_{\phi}}=\calS_{\phi}/\langle z_{\phi}\rangle$. Then the canonical epimorphism $\calS_\phi\twoheadrightarrow\overline{\calS_\phi}$ induces an inclusion
\begin{equation*}
\widehat{\overline{\calS_\phi}}\hookrightarrow\widehat{\calS_\phi}.
\end{equation*}
Here, we denote by $\widehat{A}$ the Pontryagin dual of an abelian group $A$.

\subsection{Whittaker data}\label{whit.data}
To describe our main result, we need to choose a Whittaker datum of $U(V^+)$, which is a conjugacy class of pairs $(N,\xi)$, where
\begin{itemize}
\item $N$ is the unipotent radical of a Borel subgroup of the quasi-split unitary group $U(V^+)$,
\item $\xi$ is a generic character of $N$.
\end{itemize}
When $n$ is odd, such a datum is canonical. When $n=2m$ is even, as explained in \cite{MR3202556} Section 12, it is determined by the choice of an $\Nm_{E/F}(E^\times)$-orbit of non-trivial additive characters
\[
	\begin{cases}
		\psi^E:E/F\lra\CC^\times\quad &\textit{if }\varepsilon=+1;\\
		\psi:F\lra\CC^\times\quad &\textit{if }\varepsilon=-1.
	\end{cases}
\]
According to this choice, we write
\[
	\begin{cases}
		\scrW_{\psi^E}\quad &\textit{if }\varepsilon=+1;\\
		\scrW_\psi\quad &\textit{if }\varepsilon=-1.
	\end{cases}
\]
for the corresponding Whittaker datum.\\

Assume that $\varepsilon=+1$ for a while. So that $V^+$ is a Hermitian space. By choosing a non-zero trace zero element $\delta\in E$, we can define a skew-Hermitian space $W^+=\delta\cdot V^+$, which is the space $V^+$ equipped with the skew-Hermitian form $\delta\cdot\langle\cdot,\cdot\rangle_{V^+}$. Then $U(V^+)$ and $U(W^+)$ are physically equal as subgroups of $GL(V^+)$. Let $\psi$ be a non-trivial additive character of $F$, and
\[
	\psi^{E}=\psi\left(\frac{1}{2}\Tr_{E/F}(\delta\cdot~)\right)
\]
be a character of $E$ trivial on $F$. Then, there is a Whittaker datum $\scrW_{\psi^{E}}$ of $U(V^+)$, and a Whittaker datum $\scrW_\psi$ of $U(W^+)$. We have
\[
	\scrW_{\psi^{E}}=\scrW_\psi.
\]

Now we return to the general case. Sometimes we need to consider the LLC for two (or more) unitary groups associated to spaces in a same Witt tower simultaneously, hence we must choose a Whittaker datum of each group in a compatible way. Let $\scrW$ be a Whittaker datum of the unitary group $U(V^+)$. Then, for each space $\wt{V}^+$ in the Witt tower containing $V^+$, we may choose a Whittaker datum of $U(\wt{V}^+)$ as follows. Let $\psi$ (or $\psi^E$) be a non-trivial character of $F$ (or $E/F$), such that
\[
	\scrW=\scrW_\psi\quad\left(\textit{or}\quad\scrW=\scrW_{\psi^E}\right).
\]
Then there is an obvious choice of the Whittaker datum of $U(\wt{V}^+)$, namely, the Whittaker datum associated to $\psi$ (or $\psi^E$). By abuse of notation, we shall also denote this Whittaker datum of $U(\wt{V}^+)$ by $\scrW$.

\subsection{Local factors}\label{Def.LocalFactor}
To characterize the correspondence that will be established later, we need to introduce two representation-theoretic local factors.\\

The first one is the standard $\gamma$-factor. Let $V$ be an $n$-dimensional vector space over $E$ equipped with a non-degenerate $\varepsilon$-Hermitian form, $\pi$ be an irreducible smooth representation of $U(V)$, and $\chi$ be a character of $E^\times$. Following \cite{MR2192828}, \cite{MR3166215}, one can define the standard $\gamma$-factor
\[
	\gamma(s,\pi,\chi,\psi)
\]  
by using the doubling zeta integral. We remark here that in this paper we shall use the definition in \cite{MR3166215} Section 10, which is slightly different from the definition in \cite{MR2192828} (see \cite{MR3166215} page 546 or \cite{MR4131991} Remark 5.4 for the modification of the definition and explanations). The standard $\gamma$-factors satisfy many good properties, for exmaple, ``Ten Commandments''.\\

The second one we want to introduce is the Plancherel measure. Let $\psi$ be a non-trivial additive character of $F$. Again let $V$ be an $n$-dimensional vector space over $E$ equipped with a non-degenerate $\varepsilon$-Hermitian form, $\pi$ be an irreducible smooth representation of $U(V)$, and $\tau$ be an irreducible smooth representation of $GL_k(E)$. For any $s\in\CC$, we put $\tau_s\coloneqq \tau\otimes|\det|^s$. Let $\wt V$ be the $(n+2k)$-dimensional $\varepsilon$-Hermitian space in the Witt tower containing $V$, and $P=M_PU_P$ be a maximal parabolic subgroup of $U(\wt V)$ with Levi component $M_P$ and unipotent radical $U_P$, such that
\[
	M_P\simeq GL_k(E)\times U(V).
\]
Consider the (normalized parabolic) induced representation
\begin{equation*}
\Ind_P^{U(\wt V)}(\tau_s\boxtimes\pi).
\end{equation*}
Let $\overline{P}=M_PU_{\overline{P}}$ be the parabolic subgroup of $U(\wt V)$ opposite to $P$, and $U_{\overline{P}}$ be the unipotent radical of $\overline{P}$. Fix a Haar measure $du\times d\overline{u}$ on $U_P\times U_{\overline{P}}$ as in \cite{MR3166215} Appendix B (this Haar measure depends on the choice of the additive character $\psi$). We define an intertwining operator
\[
	\calM_{\overline{P}|P}(\tau_s\boxtimes\pi):\Ind_P^{U(\wt V)}(\tau_s\boxtimes\pi)\lra\Ind_{\overline{P}}^{U(\wt V)}(\tau_s\boxtimes\pi)
\]
by (the meromorphic continuation of) the integral
\[
	\calM_{\overline{P}|P}(\tau_s\boxtimes\pi)\varPhi_s(g)=\int_{U_{\overline{P}}}\varPhi_s(ug)du.
\]
Then there exists a meromorphic function $\mu_\psi(\tau_s\boxtimes\pi)$ of $s$ such that
\[
	\calM_{P|\overline{P}}(\tau_s\boxtimes\pi)\circ\calM_{\overline{P}|P}(\tau_s\boxtimes\pi)=\mu_\psi(\tau_s\boxtimes\pi)^{-1}.
\]
In this paper, by ``Plancherel measures'', we mean the functions of the form $\mu_\psi(\tau_s\boxtimes\pi)$.\\

Given a representation $\rho$ of $WD_E$, one can define the Galois-theoretic $\gamma$-factor 
\[
	\gamma(s,\rho,\psi_E)
\]
as usual. We denote by $As^+$ the Asai representation of the $L$-group of $\Res_{E/F}GL_k$ and $As^-=As^+\otimes\omega_{E/F}$ its twist. Readers may refer to \cite{MR3202556} Section 7 for these representations.

\subsection{Main Theorem}\label{sect.desi}
Now we can formulate our desired LLC for unitary groups.
\begin{theorem}\label{MainTheorem}
There is a canonical finite-to-one surjection
\begin{equation*}
\calL:\Irr U(V^+)\sqcup\Irr U(V^-)\lra\Para(n),
\end{equation*}
where $V^+$ and $V^-$ are the $n$-dimensional $\varepsilon$-Hermitian spaces with $\epsilon(V^+)=+1$ and $\epsilon(V^-)=-1$. For each $L$-parameter $\phi$, we denote the inverse image of $\phi$ by $\Pi_\phi$, 
and call $\Pi_\phi$ the $L$-packet associated to $\phi$. For each $L$-packet $\Pi_\phi$, there is a bijection (depends on the choice of a Whittaker datum $\scrW$ of $U(V^+)$)
\begin{equation*}
\calJ_\scrW:\Pi_\phi\lra\widehat{\calS_\phi}.
\end{equation*}
We shall use $\pi(\phi,\eta)$ to denote the element in $\Pi_\phi$ corresponding to $\eta$ (with respect to $\scrW$).\\

This assignment $\pi\mapsto\left(\phi=\calL(\pi),\eta=\calJ_{\scrW}(\pi)\right)$ satisfies following properties:
\begin{enumerate}
\item The map $\calL$ preserves square-integrability.
\item The map $\calL$ preserves temperedness.
\item The map $\calL$ respects standard $\gamma$-factors, in the sense that
\begin{equation*}
\gamma(s,\pi,\chi,\psi)=\gamma(s,\phi\chi,\psi_E)
\end{equation*} 
for any $\pi\in\Irr U(V^\epsilon)$ whose parameter is $\phi$, and any character $\chi$ of $E^\times$.

\item The map $\calL$ respects Plancherel measures, in the sense that
\begin{align*}
\mu_\psi(\tau_s&\boxtimes\pi)=\gamma(s,\phi_\tau\otimes\phi^\vee,\psi_E)\cdot\gamma(-s,\phi_\tau^\vee\otimes\phi,\psi_E^{-1})\\
&\times\gamma(2s,As^{(-1)^n}\circ\phi_\tau,\psi)\cdot\gamma(-2s,As^{(-1)^n}\circ\phi_\tau^\vee,\psi^{-1})
\end{align*}
for any $\pi\in\Irr U(V^\epsilon)$ whose parameter is $\phi$, and any irreducible square-integrable representation $\tau$ of $GL_k(E)$ with $L$-parameter $\phi_\tau$. In particular, the Plancherel measures are invariants of an $L$-packet. Namely, if $\pi_1$, $\pi_2$ has the same $L$-parameter $\phi$, then we have
\begin{equation*}
\mu_\psi(\tau_s\boxtimes\pi_1)=\mu_\psi(\tau_s\boxtimes\pi_2)
\end{equation*}
for any irreducible square-integrable representation $\tau$ of $GL_k(E)$.

\item $\pi=\pi(\phi,\eta)$ is a representation of $U(V^\epsilon)$ if and only if $\eta(z_\phi)=\epsilon$.

\item Assume that $\phi$ is a tempered $L$-parameter, then there is an unique $\scrW$-generic representation of $U(V^+)$ in $\Pi_\phi$, which corresponds to the trivial character of $\calS_\phi$.

\item $\textbf{(Local Intertwining Relation)}$
Assume that 
\begin{equation*}
\phi=\phi_\tau\oplus\phi_0\oplus(\phi_\tau^c)^\vee,
\end{equation*}
where $\phi_\tau$ is an irreducible tempered representation of $WD_E$ which corresponds to an irreducible (unitary) discrete series representation $\tau$ of $GL_k(E)$ and $\phi_0$ is a tempered element in $\Phi(n-2k)$. So there is a natural embedding $\calS_{\phi_0}\hookrightarrow\calS_\phi$. Let $\pi_0=\pi(\phi_0,\eta_0)$ be an irreducible tempered representation of $U(V_0^\epsilon)$, where $V_0^\epsilon$ is the $(n-2k)$-dimensional $\varepsilon$-Hermitian space with sign $\epsilon$. There is a maximal parabolic subgroup of $U(V^\epsilon)$, say $P$, with Levi component $M$, so that
\begin{equation*}
M\simeq GL_k(E)\times U(V_0^\epsilon).
\end{equation*}
Then the induced representation $\Ind_P^{U(V^\epsilon)}(\tau\boxtimes\pi_0)$ has a decomposition
\begin{equation*}
\Ind_P^{U(V^\epsilon)}(\tau\boxtimes\pi_0)=\bigoplus_\eta\pi(\phi,\eta),
\end{equation*}
where the sum ranges over all $\eta\in\widehat{\calS_\phi}$ such that $\eta\big|_{\calS_{\phi_0}}=\eta_0$. Moreover, if $\phi_\tau$ is conjugate self-dual, let
\begin{equation*}
R(w,\tau\boxtimes\pi_0)\in\End_{U(V^\epsilon)}\left(\Ind_P^{U(V^\epsilon)}(\tau\boxtimes\pi_0)\right)
\end{equation*}
be the normalized intertwining operator to be defined later in Section \ref{IO}, where $w$ is the unique non-trivial element in the relative Weyl group for $M$. Then the restriction of $R(w,\tau\boxtimes\pi_0)$ to $\pi(\phi,\eta)$ is the scalar multiplication by
\begin{equation*}
\begin{cases}
\epsilon^k\cdot\eta(a_\tau) &\textit{if } \phi_\tau\textit{ has sign }(-1)^{n-1};\\
\epsilon^k &\textit{if } \phi_\tau\textit{ has sign }(-1)^n,
\end{cases}
\end{equation*}
where $a_\tau$ is the element in $\calS_\phi$ corresponding to $\phi_\tau$.
\item $\textbf{(Compatibility with Langlands quotients)}$
Assume that 
\begin{equation*}
\phi=(\phi_{\tau_1}|\cdot|^{s_1}\oplus\cdots\oplus\phi_{\tau_r}|\cdot|^{s_r})\oplus\phi_0\oplus\left((\phi_{\tau_1}|\cdot|^{s_1}\oplus\cdots\oplus\phi_{\tau_r}|\cdot|^{s_r})^c\right)^\vee,
\end{equation*}
where for $i=1,\cdots, r$, $\phi_{\tau_i}$ is an irreducible tempered representation of $WD_E$ which corresponds to an irreducible (unitary) discrete series representation $\tau_i$ of $GL_{k_i}(E)$, and $s_i$ is a real number such that
\begin{equation*}
s_1\geq\cdots\geq s_r>0;
\end{equation*}
$\phi_0$ is a tempered element in $\Para(n-2k)$, where $k=k_1+\cdots+k_r$. So there is a natural isomorphism $\calS_{\phi_0}\simeq\calS_\phi$. Let $\eta\in\widehat{\calS_{\phi}}$, and $\eta_0\coloneqq\eta\big|_{\calS_{\phi_0}}$. Let $\pi_0=\pi(\phi_0,\eta_0)$ be an irreducible tempered representation of $U(V_0^\epsilon)$, where $V_0^\epsilon$ is the $(n-2k)$-dimensional $\varepsilon$-Hermitian space with sign $\epsilon$. Then there is a parabolic subgroup of $U(V^\epsilon)$, say $P$, with Levi component $M$, so that
\begin{equation*}
M\simeq GL_{k_1}(E)\times\cdots\times GL_{k_r}(E)\times U(V_0^\epsilon),
\end{equation*}
and $\pi(\phi,\eta)$ is the unique irreducible quotient of the standard module 
\begin{equation*}
\Ind_P^{U(V^\epsilon)}\left(\tau_1|\det|^{s_1}\boxtimes\cdots\boxtimes\tau_r|\det|^{s_r}\boxtimes\pi_0\right).
\end{equation*}

\item If $\pi=\pi(\phi,\eta)$, and $\chi$ is a character of $E^1$, then the representation $\pi\chi\coloneqq\pi\otimes(\chi\circ\det)$ has $L$-parameter $\phi\cdot\widetilde{\chi}$ and the associated character $\eta_{\pi\chi}=\eta$, where $\widetilde{\chi}$ is the base change of $\chi$, i.e. the pull-back of $\chi$ along
\begin{align*}
E^\times&\rightarrow E^1\\
x&\mapsto x/c(x),
\end{align*}
and we use the obvious isomorphism between $\calS_\phi$ and $\calS_{\phi\cdot\widetilde{\chi}}$ to identify them.

\item If $\pi=\pi(\phi,\eta)$, then the contragredient representation $\pi^\vee$ of $\pi$ has $L$-parameter $\phi^\vee$ and associated character $\eta_{\pi^\vee}=\eta\cdot\nu$, where $\nu$ is a character of $\calS_\phi$ given by
\begin{equation*}
\nu(a)=\begin{cases}
\omega_{E/F}(-1)^{\dim \phi^a} &\textit{if }n\textit{ is even},\\
1 &\textit{if }n\textit{ is odd}
\end{cases}
\end{equation*}
for $a\in\calS_\phi$. Here we use the obvious isomorphism between $\calS_\phi$ and $\calS_{\phi^\vee}$ to identify them.
\end{enumerate}
\end{theorem}
\begin{remark}
Here, the formulations of some properties involve the LLC for a smaller unitary group $U(V_0^\epsilon)$. Hence we need to specify which Whittaker datum of $U(V_0^+)$ we are using. Notice that $V_0^+$ is in the Witt tower containing $V^+$, as explicated in Section \ref{whit.data}, the Whittaker datum $\scrW$ of $U(V^+)$ uniquely determines a Whittaker datum of $U(V_0^+)$, which we shall also denote by $\scrW$. The LLC for $U(V_0^\epsilon)$ we are using is with respect to this Whittaker datum $\scrW$.
\end{remark}

Following the method of Arthur, Mok established the LLC for quasi-split unitary groups in \cite{MR3338302} (supplemented by some results of many others):
\begin{theorem}\label{LLC.QS.Mok}
There is a canonical finite-to-one surjection
\begin{equation*}
\calL^+:\Irr(U(V^+))\lra\Para(n),
\end{equation*}
where $V^+$ is the $n$-dimensional $\varepsilon$-Hermitian space with $\epsilon(V^+)=1$. For an $L$-parameter $\phi$, let $\Pi_\phi^+$ be the inverse image of $\phi$ under $\calL^+$. For each $\Pi_\phi^+$, we have a bijection (depends on the choice of a Whittaker datum $\scrW$ of $U(V^+)$)
\begin{equation*}
\calJ^+_\scrW:\Pi_\phi^+\lra\widehat{\overline{\calS_\phi}}.
\end{equation*}
This assignment $\pi\mapsto\left(\phi=\calL^+(\pi),\eta=\calJ_{\scrW}^+(\pi)\right)$ satisfies all properties listed in Theorem \ref{MainTheorem}.
\end{theorem}

\begin{remark}\label{RMK.ExistedResults&LIRs}
\begin{enumerate}
\item There are also some existed results on the LLC for (non quasi-split) unitary groups, see \cite{kaletha2014endoscopic}, and \cite{MR3839702}. Their methods are based on trace formulas and endoscopic character identities. But we are not sure if all properties listed in Theorem \ref{MainTheorem} were verified in their works. For example, it seems to the authors that properties $(3)$, $(4)$, $(9)$, and $(10)$ are verified only for quasi-split groups, though it is expected that these properties can be verified through the endoscopic character identities. The approach in this paper is independent with these works.
\item Indeed, the LIR we formulated in Theorem \ref{MainTheorem} is the same as that in \cite{MR3573972}, but is different from the LIR formulated by Mok in \cite{MR3338302} Proposition 3.4.4, or KMSW's version in \cite{kaletha2014endoscopic} Chapter 2. There are several different points between their LIR and the LIR we formulated here: 
\begin{itemize}
	\item their LIR	is formulated for $A$-packets, rather than individual representations;
	\item their LIR is formulated in terms of distributions;
	\item the normalizing fator we used in the definition of the normalized intertwining operator is slightly different from theirs (however our normalized intertwining operator is still the same as theirs, see Remark \ref{IO.Change.Whit.Data}). 
\end{itemize}
In \cite{MR3801418}, Atobe proved that for tempered $L$-packets, the LIR we used here is a consequence of the LIR formulated by Mok/ KMSW. So we shall take it as given for quasi-split unitary groups.
\end{enumerate}
\end{remark}

We emphasize that our proof of Theorem \ref{MainTheorem} relies on Theorem \ref{LLC.QS.Mok}. Firstly we shall extend Mok's result to all odd unitary groups. Observe that when $n$ is odd, we may take $V^-=a\cdot V^+$, where $a\in F^\times\backslash\Nm_{E/F}(E^\times)$. Then $U(V^+)$ and $U(V^-)$ are physically equal as subgroups of $GL(V^+)$, and the identity map between them
induces a bijection
\begin{equation*}
\id^*:\Irr U(V^-)\lra \Irr U(V^+).
\end{equation*}
Under this identification, we can extend the map $\calL^+$ to a map
\begin{equation*}
\calL:\Irr U(V^+)\sqcup\Irr U(V^-)\lra\Para(n)
\end{equation*}
as follows: 
\begin{equation*}
\calL(\pi)=\begin{cases}
                         \calL^+(\pi) &\textit{if } \pi\in\Irr U(V^+);\\
                         \calL^+(\id^*\pi) &\textit{if } \pi\in\Irr U(V^-).

                   \end{cases}
\end{equation*}
Then for each parameter $\phi$, we have
\begin{equation*}
\Pi_\phi=\Pi_\phi^+\sqcup\Pi_\phi^-,
\end{equation*}
where $\Pi_\phi\coloneqq\calL^{-1}(\phi)$, $\Pi_\phi^+\coloneqq(\calL^+)^{-1}(\phi)$ and $\Pi_\phi^-\coloneqq(\id^*)^{-1}(\Pi_\phi^+)$. We can also extend the bijection $\calJ_{\scrW}^+$ to a bijection
\begin{equation*}
\calJ_{\scrW}:\Pi_\phi\lra\widehat{\calS_\phi}
\end{equation*} 
by letting
\begin{equation*}
\calJ_{\scrW}(\pi)=\begin{cases}
                         \calJ_{\scrW}^+(\pi) &\textit{if } \pi\in\Pi_\phi^+;\\
                         \calJ_{\scrW}^+(\id^*\pi)\cdot\eta_{-} &\textit{if } \pi\in\Pi_\phi^-.

                   \end{cases}
\end{equation*}
where $\eta_-$ is a character of $\calS_\phi$ given by
\begin{equation*}
\eta_-(a)=(-1)^{\dim \phi^a}
\end{equation*}
for $a\in\calS_\phi$. One can easily check that $\calL$ and $\calJ_{\scrW}$ give us what we want:
\begin{theorem}\label{TheoremOdd}
Theorem \ref{MainTheorem} holds for $n$ odd.
\end{theorem}

Hence in the rest of this paper, we will focus on proving Theorem \ref{MainTheorem} for $n$ even.

\section{Theta correspondence}\label{weilrep}
In this section, we recall the notion of the Weil representation and local theta correspondence.

\subsection{Weil representations}\label{AuxiliaryDatum}
Let $V$ be a Hermitian space and $W$ a skew-Hermitian space. To consider the theta correspondence for the reductive dual pair $U(V)\times U(W)$, one requires some additional data:
\begin{itemize}
\item a non-trivial additive character $\psi$ of $F$;
\item a pair of characters $\chi_V$ and $\chi_W$ of $E^\times$ such that 
\begin{equation*}
\chi_V\big|_{F^\times}=\omega_{E/F}^{\dim V}\quad\textit{and}\quad\chi_W\big|_{F^\times}=\omega_{E/F}^{\dim W};
\end{equation*}
\item a trace zero element $\delta\in E^\times$.
\end{itemize}
To elaborate, the tensor product $V\otimes W$ has a natural symplectic form defined by
\begin{equation*}
\langle v_1\otimes w_1, v_2\otimes w_2\rangle=\Tr_{E/F}(\langle v_1,v_2\rangle_V\cdot\langle w_1,w_2\rangle_W).
\end{equation*}
Then there is a natural map
\begin{equation*}
U(V)\times U(W)\lra Sp(V\otimes W).
\end{equation*}
One has the metaplectic $S^1$-cover $Mp(V\otimes W)$ of $Sp(V\otimes W)$, and the character $\psi$ (together with the form $\langle\cdot,\cdot\rangle$ on $V\otimes W$) determines a Weil representation $\omega_\psi$ of $Mp(V\otimes W)$. The datum $\underline{\psi}\coloneqq(\psi,\chi_V,\chi_W,\delta)$ then allows one to specify a splitting of the metaplectic cover over $U(V)\times U(W)$. In \cite{MR1286835}, \cite{MR1327161}, it is showed that this splitting in fact does not depend on the choice of $\delta$. Hence, we have a Weil representation $\omega_{\underline{\psi},V,W}$ of $U(V)\times U(W)$. The Weil representation $\omega_{\underline{\psi},V,W}$ depends only on the orbit of $\psi$ under $\Nm_{E/F}E^\times$.

\subsection{Local theta correspondence}
Given an irreducible representation $\pi$ of $U(W)$, the maximal $\pi$-isotypic quotient of $\omega_{\underline{\psi},V,W}$ is of the form 
\begin{equation*}
\Theta_{\underline{\psi},V,W}(\pi)\boxtimes\pi
\end{equation*}
for some smooth representation $\Theta_{\underline{\psi},V,W}(\pi)$ of $U(V)$ of finite length. By the Howe duality, which was proved by Waldspurger \cite{MR1159105} for $p\neq2$ and by Gan-Takeda \cite{MR3502978}, \cite{MR3454380} for any $p$, we have
\begin{itemize}
\item The maximal semi-simple quotient $\theta_{\underline{\psi},V,W}(\pi)$ of $\Theta_{\underline{\psi},V,W}(\pi)$ is irreducible if $\Theta_{\underline{\psi},V,W}(\pi)$ is non-zero;
\item If $\pi_1$ and $\pi_2$ are irreducible smooth representations of $U(W)$, such that both $\theta_{\underline{\psi},V,W}(\pi_1)$ and $\theta_{\underline{\psi},V,W}(\pi_2)$ are non-zero. Assume that $\pi_1\not\simeq\pi_2$. Then $\theta_{\underline{\psi},V,W}(\pi_1)\not\simeq\theta_{\underline{\psi},V,W}(\pi_2)$.
\end{itemize}

In this paper, we use the theta correspondence for $U(V)\times U(W)$ with 
\begin{equation*}
|\dim V-\dim W|\leq1
\end{equation*}
to construct the LLC for even unitary groups. In our proofs, we shall use some results in the context of the theta correspondence from \cite{MR3166215}. We emphasize that the proofs of those results are independent of the LLC for unitary groups.\\

Here we give a generalization of \cite{MR3166215} Proposition C.4, which will be frequently used in later proofs.
\begin{lemma}\label{theta.ind}
Let $l=\dim W-\dim V$. Assume that $l=-1$. Let $\pi$ be an irreducible tempered representation of $U(W)$ such that
\begin{equation*}
\pi\subset\Ind_Q^{U(W)}(\tau\chi_V\boxtimes\pi_0),
\end{equation*}
where $Q$ is a maximal parabolic subgroup of $U(W)$ with Levi component $GL_k(E)\times U(W_0)$, $\tau$ is an irreducible (unitary) discrete series representation of $GL_k(E)$ and $\pi_0$ is an irreducible tempered representation of $U(W_{0})$. Let
\[
	m_Q(\pi)=\dim\Hom_{U(W)}\left(\pi,\Ind_Q^{U(W)}(\tau\chi_V\boxtimes\pi_0)\right),
\]
and 
\[
	m_P\left(\theta_{\underline{\psi},V,W}(\pi)\right)=\dim\Hom_{U(V)}\left(\theta_{\underline{\psi},V,W}(\pi),\Ind_P^{U(V)}(\tau\chi_W\boxtimes\theta_{\underline{\psi},V_{0},W_{0}}(\pi_0))\right).
\]
Then we have
\[
	m_Q(\pi)\leq m_P\left(\theta_{\underline{\psi},V,W}(\pi)\right).
\]
\end{lemma}
\begin{proof}
The proof is almost the same as that of \cite{MR3166215} Proposition C.1. The only difference is that we count the multiplicities. Readers may also refer to \cite{MR4308934} Lemma 3.5 and Lemma 8.2.
\end{proof}

\section{Constructions}\label{LLC.Construction}
In this section, we will construct an LLC for even unitary groups. We will first construct such a correspondence for tempered representations, and then extend the construction to non-tempered representations based on the tempered case. Several properties listed in Theorem \ref{MainTheorem} will be proved along the way.\\

Before we start, we set up some notations here. For $\epsilon=\pm1$, let $\calV^\epsilon$ be the Witt tower of Hermitian spaces which consists of all $V_{2n+1}^\epsilon$, where $V_{2n+1}^\epsilon$ is the $(2n+1)$-dimensional Hermitian space over $E$ with sign $\epsilon$. Similarly, let $\calW^\epsilon$ be the Witt tower of skew-Hermitian spaces which consists of all $W_{2n}^\epsilon$, where $W_{2n}^\epsilon$ is the $2n$-dimensional skew-Hermitian space over $E$ with sign $\epsilon$. Let
\[
	\underline{\psi}=(\psi,\chi_V,\chi_W,\delta)
\]
be a tuple of data described in Section \ref{AuxiliaryDatum}. Let $W$ be an even dimensional skew-Hermitian space. For an irreducible smooth representation $\pi$ of $U(W)$, we will use $\theta_{\underline{\psi},2n+1}^\epsilon(\pi)$ to denote the theta lift of $\pi$ to $V_{2n+1}^\epsilon$, with respect to the datum $\underline{\psi}$. Similarly, let $V$ be an odd dimensional Hermitian space. For an irreducible smooth representation $\sigma$ of $U(V)$, we will use $\theta_{\underline{\psi},2n}^\epsilon(\sigma)$ to denote the theta lift of $\sigma$ to $W_{2n}^\epsilon$, with respect to the datum $\underline{\psi}$.

\subsection{Construction of $\calL$}
First of all, we attach $L$-parameters to irreducible tempered representations of even unitary groups. We shall use two steps to achieve this purpose. In the first step, for each tuple of data $\underline{\psi}=(\psi,\chi_V,\chi_W,\delta)$, we construct a map
\begin{equation*}
\calL_{\underline{\psi}}:\Irrt U(W_{2n}^+)\sqcup \Irrt U(W_{2n}^-)\lra\Parat(2n).
\end{equation*}
Then in the second step, we show that indeed $\calL_{\underline{\psi}}$ is independent of the choice of ${\underline{\psi}}$, so we get the desired map $\calL$.\\

In this subsection we do the first step. Fix a tuple of data $\underline{\psi}=(\psi,\chi_V,\chi_W,\delta)$. Given $\pi\in\Irrt U(W_{2n}^\epsilon)$, consider its theta lifts to $U(V_{2n+1}^+)$ and $U(V_{2n-1}^-)$:
\begin{equation*}
\begindc{\commdiag}[75]
\obj(8,4)[b]{$U(V_{2n+1}^+)$}
\obj(0,0)[x]{$U(W_{2n}^\epsilon)$}
\obj(-8,-4)[c]{$U(V_{2n-1}^-)$}
\obj(8,1)[i]{$\theta_{\underline{\psi},2n+1}^+(\pi)$}
\obj(0,-3)[y]{$\pi$}
\obj(-8,-7)[j]{$\theta_{\underline{\psi},2n-1}^-(\pi)$}
\mor{b}{x}{}[\atleft,\solidline]
\mor{c}{x}{}[\atleft,\solidline]
\mor{y}{i}{}[\atleft,\aplicationarrow]
\mor{y}{j}{}[\atleft,\aplicationarrow]
\enddc
\end{equation*}
By the conservation relation (see \cite{MR2885581}, \cite{MR3369906}), we know that exactly one of these two representations is non-zero. Also, by \cite{MR3166215} Proposition C.4, this non-zero representation is also tempered.\\

\underline{CASE I:} If $\sigma\coloneqq\theta_{\underline{\psi},2n+1}^+(\pi)\neq 0$, then we have:
\begin{lemma}
Let $\phi_\sigma$ be the $L$-parameter of $\sigma$. Then we have $\chi_W\subset\phi_\sigma$.
\end{lemma}
\begin{proof}
By the Howe duality, $\theta_{\underline{\psi},2n}^\epsilon(\sigma)=\pi$ is non-zero. Hence by \cite{MR3166215} Proposition 11.2, $\gamma(s,\sigma,\chi_W^{-1},\psi)$ has a pole at $s=1$. Applying the LLC for odd unitary groups, we get
\begin{equation*}
\gamma(s,\sigma,\chi_W^{-1},\psi)=\gamma(s,\phi_\sigma\chi_W^{-1},\psi_E).
\end{equation*}
Since $\sigma$ is tempered, $\phi_\sigma$ is also tempered. This implies that $\phi_\sigma\chi_W^{-1}$ contains a trivial representation. Hence we conclude $\chi_W\subset\phi_\sigma$ as desired.
\end{proof}
In this case, we define $\calL_{\underline{\psi}}(\pi)$ to be
\begin{equation*}
\phi\coloneqq(\phi_\sigma -\chi_W)\chi_W^{-1}\chi_V.
\end{equation*}

\underline{CASE II:} If $\sigma\coloneqq\theta_{\underline{\psi},2n-1}^-(\pi)\neq0$. Let $\phi_\sigma$ be the $L$-parameter of $\sigma$. In this case, we simply define $\calL_{\underline{\psi}}(\pi)$ to be 
\begin{equation*}
\phi\coloneqq \phi_\sigma\chi_W^{-1}\chi_V\oplus\chi_V.
\end{equation*}

Notice that in either case, $\phi$ is a tempered parameter. Thus we get a map
\begin{equation*}
\calL_{\underline{\psi}}:\Irrt U(W_{2n}^+)\sqcup \Irrt U(W_{2n}^-)\lra\Parat(2n).
\end{equation*}

\begin{lemma}\label{WeakPrasadConj.QS}
Let $\pi$ be an irreducible tempered representation of $U(W_{2n}^+)$. Let $\phi$ be the $L$-parameter of $\pi$ in the sense of Mok's LLC for quasi-split unitary groups, i.e. $\phi=\calL^+(\pi)$. Then we have
\[
	\calL_{\underline{\psi}}(\pi)=\phi.
\]
\end{lemma}
\begin{proof}
See \cite{MR3166215} Page 652. (If we only consider the case $\pi\in\Irrt U(W_{2n}^+)$, the proof in \cite{MR3166215} Page 652 will only involve Mok's LLC for quasi-split unitary groups, without refering any non quasi-split unitary groups.)
\end{proof}

\subsection{Independency}
In the previous subsection we have constructed the map $\calL_{\underline{\psi}}$. Now we do the second step.
\begin{lemma}\label{PreserveLocalFactor.1st}
Let $\pi\in\Irrt U(W_{2n}^\epsilon)$ and $\phi=\calL_{\underline{\psi}}(\pi)$.
\begin{enumerate}
\item For any character $\chi$ of $E^\times$, we have
\begin{equation*}
\gamma(s,\pi,\chi,\psi)=\gamma(s,\phi\chi,\psi_E).
\end{equation*}
\item For any irreducible square-integrable representation $\tau$ of $GL_k(E)$ with $L$-parameter $\phi_\tau$, we have
\begin{align*}
\mu_\psi(\tau_s&\boxtimes\pi)=\gamma(s,\phi_\tau\otimes\phi^\vee,\psi_E)\cdot\gamma(-s,\phi_\tau^\vee\otimes\phi,\psi_E^{-1})\\
&\times\gamma(2s,As^{(-1)^n}\circ\phi_\tau,\psi)\cdot\gamma(-2s,As^{(-1)^n}\circ\phi_\tau^\vee,\psi^{-1}).
\end{align*}
\end{enumerate}
\end{lemma}
\begin{proof}
We only prove the first statement here. The proof of the second is similar. According to our construction, we need to consider two cases.\\

\underline{CASE I:} Suppose that $\sigma\coloneqq\theta_{\underline{\psi},2n+1}^+(\pi)\neq 0$. Then for any character $\chi$ of $E^\times$, by \cite{MR3166215} Theorem 11.5, we have
\[
	\frac{\gamma(s,\sigma,\chi\chi_W^{-1}\chi_V,\psi)}{\gamma(s,\pi,\chi,\psi)}=\gamma\left(s,\chi\chi_V,\psi_E\right).
\]
Let $\phi_\sigma$ be the $L$-parameter of $\sigma$. It follows from our construction that
\[
	\phi_\sigma=\phi\chi_V^{-1}\chi_W\oplus\chi_W.
\]
By the LLC for odd unitary groups, we have
\begin{align*}
	\gamma(s,\sigma,\chi\chi_W^{-1}\chi_V,\psi)&=\gamma(s,\phi_\sigma\chi\chi_W^{-1}\chi_V,\psi_E)\\
	&=\gamma(s,\phi\chi,\psi_E)\cdot\gamma\left(s,\chi\chi_V,\psi_E\right).
\end{align*}
Combining these equalities, we get
\[
	\gamma(s,\pi,\chi,\psi)=\gamma(s,\phi\chi,\psi_E).
\]
Hence the first statement holds in this case.\\

\underline{CASE II:} Suppose that $\sigma\coloneqq\theta_{\underline{\psi},2n-1}^-(\pi)\neq0$. In this case the desired formula also follows from a similar computation. We omit the details here.
\end{proof}

\begin{corollary}\label{ParameterIndependent}
The map $\calL_{\underline{\psi}}$ is independent of the choice of $\underline{\psi}$.
\end{corollary}
\begin{proof}
Assume that $\underline{\psi'}=(\psi', \chi_V',\chi_W',\delta)$ is another tuple of data. We define the map
\begin{equation*}
\calL_{\underline{\psi'}}:\Irrt U(W_{2n}^+)\sqcup \Irrt U(W_{2n}^-)\lra\Parat(2n)
\end{equation*}
in a similar procedure. By Lemma \ref{WeakPrasadConj.QS}, the restrictions of both $\calL_{\underline{\psi}}$ and $\calL_{\underline{\psi'}}$ to $\Irrt U(W_{2n}^+)$ coincide with $\calL^+$, i.e.
\begin{equation*}
\calL_{\underline{\psi}}\Big|_{\Irrt U(W_{2n}^+)}=\calL^+=\calL_{\underline{\psi'}}\Big|_{\Irrt U(W_{2n}^+)}.
\end{equation*}
Now given any irreducible tempered representation $\pi$ of $U(W_{2n}^\epsilon)$, we can find a representation $\pi'\in\Irrt U(W_{2n}^+)$, such that 
\begin{equation*}
\calL^+(\pi')=\calL_{\underline{\psi}}(\pi')=\calL_{\underline{\psi}}(\pi).
\end{equation*}
Hence by Lemma \ref{PreserveLocalFactor.1st}, for all $k\geq1$, and all irreducible square-integrable representation $\tau$ of $GL_k(E)$, we have
\begin{equation*}
\mu_{\psi'}(\tau_s\boxtimes\pi')=C_{\psi',\psi,2n,k}\cdot\mu_\psi(\tau_s\boxtimes\pi')=C_{\psi',\psi,2n,k}\cdot\mu_\psi(\tau_s\boxtimes\pi)=\mu_{\psi'}(\tau_s\boxtimes\pi).
\end{equation*}
where $C_{\psi',\psi,2n,k}$ is a constant only depends on $\psi'$, $\psi$, $2n$ and $k$. This equality, together with Lemma \ref{PreserveLocalFactor.1st} and \cite{MR3573972} Lemma A.6, implies that
\begin{equation*}
\calL^+(\pi')=\calL_{\underline{\psi'}}(\pi')=\calL_{\underline{\psi'}}(\pi).
\end{equation*}
Hence $\calL_{\underline{\psi}}(\pi)=\calL_{\underline{\psi'}}(\pi)$. In other words, $\calL_{\underline{\psi}}$ is independent of the choice of $\underline{\psi}$.
\end{proof}

After proving that the map $\calL_{\underline{\psi}}$ is indeed independent of the choice of $\underline{\psi}$, we will denote the map abstractly by $\calL$. For an irreducible tempered representation $\pi$ of $U(W_{2n}^\epsilon)$, we call $\calL(\pi)$ the $L$-parameter of $\pi$. For a tempered $L$-parameter $\phi$, we let $\Pi_\phi$ be the fiber $\calL^{-1}(\phi)$, and call it the $L$-packet of $\phi$. For $\epsilon=\pm1$, we also let $\Pi_\phi^\epsilon=\Pi_\phi\cap\Irr U(W_{2n}^\epsilon)$. Combining Lemma \ref{PreserveLocalFactor.1st} and Corollary \ref{ParameterIndependent}, we get
\begin{corollary}\label{PreserveLocalFactor.2nd}
For tempered representations, the map $\calL$ respects standard $\gamma$-factors and Plancherel measures.
\end{corollary}

\subsection{Counting Sizes of Packets}\label{LLC.Construction.J}
Our next goal is to attach a character of component group to each irreducible tempered representation of even unitary groups. To do this, we need some preparations. In this subsection we consider the behaviour of $L$-parameters under the local theta correspondence and count the sizes of $L$-packets for even unitary groups. In this subsection when we talk about representations of odd unitary groups, the $L$-parameter of a representation is in the sense of Theorem \ref{TheoremOdd}; whereas when we talk about representations of even unitary groups, the $L$-parameter of a tempered representation is in the sense of $\calL$.\\

To define the map $\calJ_\scrW$, we need to fix an Whittaker datum $\scrW$ of $U(W_{2n}^+)$. As explained in Section \ref{whit.data}, once we fix the Whittaker datum $\scrW$, we may pick up a non-trivial additive character $\psi$ of $F$, such that 
\[
	\scrW=\scrW_\psi.
\]
We fix a pair of characters $(\chi_V,\chi_W)$ of $E^\times$ and a trace zero element $\delta\in E^\times$ as in Section \ref{AuxiliaryDatum}. If there is no further explanation, the theta lifts used in the rest of this section will be with respect to the datum
\[
	\underline{\psi}=(\psi,\chi_V,\chi_W,\delta).
\]
We shall use $\calL_{\underline{\psi}}$ to `realize' the map $\calL$. For simplicity, we shall drop the subscript ``$\underline{\psi}$'' and just denote $\theta_{\underline{\psi},*}^\pm$ by $\theta_{*}^\pm$. In the rest of this section, if $\rho$ is an irreducible smooth representation of some group $G$, we shall use the symbol $\phi_\rho$ to denote the $L$-parameter of $\rho$. If $G$ is an odd unitary group, we shall also use the symbol $\eta_\rho$ to denote the character of $\calS_{\phi_\rho}$ associated to $\rho$.
\begin{lemma}\label{Compatible.Parameter}
\begin{enumerate}
\item If $\pi\in\Irrt U(W_{2n}^\epsilon)$, such that ${\sigma}\coloneqq\theta_{2n+1}^{\epsilon'} (\pi)\neq 0$. Then 
\begin{equation*}
\phi_{\sigma}=\phi_\pi\chi_V^{-1}\chi_W \oplus\chi_W.
\end{equation*}
\item Similarly, if $\sigma\in\Irrt U(V_{2m-1}^\epsilon)$, such that ${\pi}\coloneqq\theta_{2m}^{\epsilon'} (\sigma)\neq 0$. Then 
\begin{equation*}
\phi_{{\pi}}=\phi_\sigma\chi_W^{-1}\chi_V \oplus\chi_V.
\end{equation*}
\end{enumerate}
\end{lemma}
\begin{proof}
With Lemma \ref{PreserveLocalFactor.1st} at hand, we can appeal to the same argument of \cite{MR3166215} Page 652 to prove this lemma. We omit the details here.
\end{proof}
As a consequence, we deduce
\begin{corollary}\label{Lift.Up}
\begin{enumerate}
\item Let $\pi\in\Irrt U(W_{2n}^\epsilon)$. If $\chi_V\not\subset\phi_\pi$, then $\theta_{2n-1}^\pm(\pi)=0$. Hence by the conservation relation, both $\theta_{2n+1}^+(\pi)$ and $\theta_{2n+1}^-(\pi)$ are non-zero.
\item Similarly, let $\sigma\in\Irrt U(V_{2m-1}^\epsilon)$. If $\chi_W\not\subset\phi_\sigma$, then $\theta_{2m-2}^\pm(\sigma)=0$. Hence by the conservation relation, both $\theta_{2m}^+(\sigma)$ and $\theta_{2m}^-(\sigma)$ are non-zero.
\end{enumerate}
\end{corollary}

\begin{lemma}\label{Lift.Down}
\begin{enumerate}
\item Let $\pi\in\Irrt U(W_{2n}^\epsilon)$. If $\chi_V\subset\phi_\pi$, then exactly one of $\theta_{2n-1}^+(\pi)$ and $\theta_{2n-1}^-(\pi)$ is non-zero.
\item Similarly, let $\sigma\in\Irrt U(V_{2m-1}^\epsilon)$. If $\chi_W\subset\phi_\sigma$, then exactly one of $\theta_{2m-2}^+(\sigma)$ and $\theta_{2m-2}^-(\sigma)$ is non-zero.
\end{enumerate}
\end{lemma}
\begin{proof}
We only prove the first statement here. The proof of the second statement is similar. We shall prove this by counting fibers of the map $\calL=\calL_{\underline{\psi}}$.
We define a map 
\begin{equation*}
\theta_{2n+1}:\Irr U(W_{2n}^+)\sqcup \Irr U(W_{2n}^-)\lra \Irr U(V_{2n+1}^+)\sqcup \Irr U(V_{2n+1}^-)
\end{equation*}
as follows:
\begin{equation*}
\pi'\mapsto\begin{cases}
                         \theta_{2n+1}^+(\pi') &\textit{if } \theta_{2n+1}^+(\pi')\neq 0;\\
                         \theta_{2n+1}^-(\pi') &\textit{otherwise}.

                   \end{cases}
\end{equation*}
By the Howe duality and the conservation relation, this map is well-defined and injective. For each tempered $L$-parameter $\phi$, by Lemma \ref{Compatible.Parameter}, the restriction of this map to the $L$-packet $\Pi_{\phi}$ gives an injection
\begin{equation*}
\theta_{2n+1}:\Pi_{\phi}\hookrightarrow\Pi_{\phi^+},
\end{equation*}
where $\phi^+\coloneqq\phi\chi_V^{-1}\chi_W\oplus\chi_W$.\\

Now we let $\phi=\phi_\pi$. By our assumption, $\chi_V\subset\phi$. Let $\phi^-\coloneqq(\phi-\chi_V)\chi_V^{-1}\chi_W$.\\

\underline{CASE I:} If $2\chi_V\subset\phi$, then $\chi_W\subset\phi^-$. Similarly we can define another map
\begin{equation*}
\theta_{2n}:\Irr U(V_{2n-1}^+)\sqcup \Irr U(V_{2n-1}^-)\lra \Irr U(W_{2n}^+)\sqcup \Irr U(W_{2n}^-)
\end{equation*}
in the same way as $\theta_{2n+1}$. Again by Lemma \ref{Compatible.Parameter}, the restriction of this map to the packet $\Pi_{\phi^-}$ gives an injection
\begin{equation*}
\theta_{2n}:\Pi_{\phi^-}\hookrightarrow\Pi_{\phi}.
\end{equation*}
Hence we have
\begin{equation*}
|\Pi_{\phi^-}|\leq|\Pi_{\phi}|\leq|\Pi_{\phi^+}|.
\end{equation*}
But in this case, $\calS_{\phi^-}\simeq\calS_\phi\simeq\calS_{\phi^+}$, by using the LLC for odd unitary groups, we get
\begin{equation*}
|\Pi_{\phi^-}|=|\widehat{\calS_{\phi^-}}|=|\widehat{\calS_{\phi^+}}|=|\Pi_{\phi^+}|.
\end{equation*}
This implies that $\theta_{2n}$ is surjective. Hence in this case the lemma holds.\\

\underline{CASE II:} If $2\chi_V\not\subset\phi$, then $\chi_W\not\subset\phi^-$. We can define a map
\begin{equation*}
\theta_{2n}^+\sqcup\theta_{2n}^-: \Pi_{\phi^-}\sqcup\Pi_{\phi^-}\lra\Pi_{\phi}
\end{equation*}
by
\begin{equation*}
\begin{cases}
\sigma\mapsto\theta_{2n}^+(\sigma) &\textit{for } \sigma\textit{ in the first copy of }\Pi_{\phi^-};\\
\sigma\mapsto\theta_{2n}^-(\sigma) &\textit{for } \sigma\textit{ in the second copy of }\Pi_{\phi^-}.
\end{cases}
\end{equation*}
Again, by the Howe duality, the conservation relation, and Corallary \ref{Lift.Up}, it's easy to see that this map is well-defined and injective. Thus we have
\begin{equation*}
2|\Pi_{\phi^-}|\leq|\Pi_{\phi}|\leq|\Pi_{\phi^+}|.
\end{equation*}
Also, in this case, 
\begin{equation*}
\calS_{\phi^+}\simeq\calS_\phi\simeq\calS_{\phi^-}\oplus(\ZZ/2\ZZ)e,
\end{equation*} 
where $e$ is the element in $\calS_\phi$ corresponding to $\chi_V\subset\phi$. By using the LLC for odd unitary groups, we get
\begin{equation*}
2|\Pi_{\phi^-}|=2|\widehat{\calS_{\phi^-}}|=|\widehat{\calS_{\phi^+}}|=|\Pi_{\phi^+}|.
\end{equation*}
This implies that $\theta_{2n}^+\sqcup\theta_{2n}^-$ is surjective. Hence in this case the lemma also holds.
\end{proof}

As a consequence of this Lemma, we can compute the sizes of $L$-packets.
\begin{corollary}\label{Size.Packet}
Let $\phi\in\Parat(2n)$. Then the size of the $L$-packet $\Pi_{\phi}$ is exactly the same as the size of $\wh{\calS_\phi}$. In particular, the packet is non-empty. 
\end{corollary}
\begin{proof}
The case when $\chi_V\subset\phi$ follows directly from the proof of Lemma \ref{Lift.Down}. So it is sufficient to prove the case when $\chi_V\not\subset\phi$. Similar to the proof of Lemma \ref{Lift.Down}, the theta lift gives us injections
\begin{equation*}
\theta_{2n+1}^\epsilon:\Pi_{\phi}\hookrightarrow\Pi_{\phi^+}^\epsilon
\end{equation*}
for $\epsilon=\pm1$. The Lemma \ref{Lift.Down} tells us these injections are also surjective. Notice that in this case, we have
\[
	\calS_{\phi^+}\simeq\calS_\phi\oplus(\ZZ/2\ZZ)e,
\]
where $e$ is the element in $\calS_{\phi^+}$ corresponding to $\chi_W\subset\phi^+$. This induces an isomorphism
\[
	\calS_\phi\hookrightarrow\calS_{\phi^+}\twoheadrightarrow\overline{\calS_{\phi^+}}.
\]
Hence by the LLC for odd unitary groups, we conclude that
\[
	|\Pi_\phi|=|\Pi_{\phi^+}^\epsilon|=|\widehat{\overline{\calS_{\phi^+}}}|=|\widehat{\calS_\phi}|
\]
as desired.
\end{proof}

\subsection{Construction of $\calJ_\scrW$}\label{LLC.Construction.J}
Now given a tempered parameter $\phi\in\Parat(2n)$, we have shown the size of the $L$-packet $\Pi_{\phi}$ is the same as $\widehat{\calS_\phi}$. Next, we are going to define the bijection
\begin{equation*}
\calJ_{\scrW}:\Pi_{\phi}\lra\widehat{\calS_\phi}.
\end{equation*}
We separate the construction into two cases.\\

\underline{CASE I:} If $\chi_V\not\subset\phi$, then by Corallary \ref{Lift.Up}, we have $\sigma\coloneqq\theta_{2n+1}^+(\pi)\neq0$. And by our construction, $\phi_\sigma=\phi\chi_V^{-1}\chi_W\oplus\chi_W$. Therefore
\begin{equation*}
\calS_{\phi_\sigma}\simeq\calS_\phi\oplus(\ZZ/2\ZZ)e,
\end{equation*} 
where $e$ is the element in $\calS_{\phi_\sigma}$ corresponding to $\chi_W\subset\phi_\sigma$. This induces an isomorphism
\begin{equation*}
\iota:\calS_\phi\hookrightarrow\calS_{\phi_\sigma}\twoheadrightarrow\overline{\calS_{\phi_\sigma}}.
\end{equation*}
In this case we define the character $\eta\in\widehat{\calS_\phi}$ associated to $\pi$ to be
\begin{equation*}
\eta\coloneqq\eta_\sigma\big|_{\calS_\phi}.
\end{equation*}

\underline{CASE II:} If $\chi_V\subset\phi$, then by the Lemma \ref{Lift.Down}, there exists an unique $\epsilon'$, such that $\theta_{2n-1}^{\epsilon'}(\pi)$ is non-zero, hence $\sigma\coloneqq\theta_{2n+1}^{\epsilon'}(\pi)$ is also non-zero by the persistence of theta lifts. According to Lemma \ref{Compatible.Parameter}, $\phi_\sigma=\phi\chi_V^{-1}\chi_W\oplus\chi_W$. Thus
\begin{equation*}
\calS_\phi\simeq\calS_{\phi_\sigma}.
\end{equation*}
In this case we define the character $\eta\in\widehat{\calS_\phi}$ associated to $\pi$ to be
\begin{equation*}
\eta\coloneqq\eta_\sigma\big|_{\calS_\phi}.
\end{equation*}

By the LLC for odd unitary groups, the Howe duality, and Corollary \ref{Size.Packet}, it is easy to check that the assignment constructed here gives a bijection between $\Pi_{\phi}$ and $\wh{\calS_\phi}$. 

\subsection{From tempered to non-tempered}
So far, we have attached $L$-parameters and characters of component groups for all irreducible tempered representations of $U(W_{2n}^\epsilon)$. Next, for an irreducible non-tempered representation $\pi$ of $U(W_{2n}^\epsilon)$, we shall attach an $L$-parameter and a character of component group to it. Readers may also refer to \cite{MR3271238}.\\

Let $\pi$ be an irreducible smooth representation of $U(W_{2n}^\epsilon)$. By Langlands' classification for $p$-adic groups \cite{MR507262}, \cite{MR2050093}, we know that $\pi$ is the unique irreducible quotient of a standard module
\begin{equation*}
\Ind_P^{U(W_{2n}^\epsilon)}\left(\tau_1|\det|^{s_1}\boxtimes\cdots\boxtimes\tau_r|\det|^{s_r}\boxtimes\pi_0\right),
\end{equation*}
where $P$ is a parabolic subgroup of $U(W_{2n}^\epsilon)$, with a Levi component
\begin{equation*}
M\simeq GL_{k_1}(E)\times\cdots\times GL_{k_r}(E)\times U(W_{2n}^\epsilon),\quad k=k_1+\cdots+k_r;
\end{equation*}
${\tau_i}$ is an irreducible (unitary) square-integrable representation of $GL_{k_i}(E)$, $s_i$ is a real number such that
\begin{equation*}
s_1\geq\cdots\geq s_r>0;
\end{equation*}
and $\pi_0$ is an irreducible tempered representation of $U(W_{2n-2k}^\epsilon)$. Let $\phi_{\tau_i}$ be the $L$-parameter of $\tau_i$, and $\pi_0=\pi(\phi_0,\eta_0)$. We define the $L$-parameter of $\pi$ to be
\begin{equation*}
\phi=(\phi_{\tau_1}|\cdot|^{s_1}\oplus\cdots\oplus\phi_{\tau_r}|\cdot|^{s_r})\oplus\phi_0\oplus\left((\phi_{\tau_1}|\cdot|^{s_1}\oplus\cdots\oplus\phi_{\tau_r}|\cdot|^{s_r})^c\right)^\vee.
\end{equation*}
Notice that $\calS_\phi\simeq\calS_{\phi_0}$. Via this natural identification, we define the character in $\widehat{\calS_\phi}$ associated to $\pi$ to be
\begin{equation*}
\eta=\eta_0.
\end{equation*}
Since the datum $(P,\tau_i,s_i,\pi_0)$ is uniquely determined by $\pi$ up to Weyl group conjugate, $\phi$ and $\eta$ are well-defined.\\

From now on, we shall use $\pi(\phi,\eta)$ to denote the element in $\Pi_\phi$ corresponding to $\eta$. It follows directly from our construction that 
\begin{proposition}
The LLC we constructed for even unitary groups is compatible with Langlands quotients.
\end{proposition}

An easy computation shows that
\begin{proposition}\label{PreserveLocalFactor.Fin}
The LLC we constructed for even unitary groups respects standard $\gamma$-factors and Plancherel measures.
\end{proposition}
\begin{proof}
We have proved this proposition for tempered representations. The general case follows from Lemma \ref{PreserveLocalFactor.2nd} and multiplicativity of standard $\gamma$-factors \& the Plancherel measures (see \cite{MR3166215} Section 10.2, and Appendix B.5).
\end{proof}

\subsection{Preservation}
In this section we prove two further properties of the map $\calL$.

\begin{proposition}\label{Preserve.S-I}
The map $\calL$ preserves square-integrability.
\end{proposition}
\begin{proof}
Let $\pi$ be an irreducible smooth representation of $U(W_{2n}^\epsilon)$, and $\phi$ be the $L$-parameter of $\pi$. We first prove that if $\pi$ is square-integrable, then $\phi$ is square-integrable. We divide this into two cases.\\

\underline{CASE I:} If $\chi_V\not\subset\phi$, then by the Corollary \ref{Lift.Up}, $\theta_{2n-1}^+(\pi)=0$ and $\sigma\coloneqq\theta_{2n+1}^+(\pi)\neq0$. Hence by \cite{MR3166215} Corollary C.3, $\sigma$ is also square-integrable. The LLC for odd unitary groups then implies that
\[
	\phi_\sigma=\phi\chi_V^{-1}\chi_W\oplus\chi_W
\] 
is square-integrable. Thus $\phi$ is also square-integrable.\\

\underline{CASE II:} If $\chi_V\subset\phi$, then by Lemma \ref{Lift.Down}, there exists $\epsilon'\in\{\pm1\}$, such that $\sigma\coloneqq\theta_{2n-1}^{\epsilon'}(\pi)\neq0$. Hence by \cite{MR3166215} Corollary C.3, $\sigma$ is also square-integrable. The LLC for odd unitary groups then implies that $\phi_\sigma$ is square-integrable. We claim that $\chi_W\not\subset\phi_\sigma$. Indeed, suppose on the contrary that $\chi_W\subset\phi_\sigma$, by Lemma \ref{Lift.Down} and the conservation relation, we must have
\[
	\theta_{2n-2}^\epsilon(\sigma)\neq0.
\] 
Again by \cite{MR3166215} Corollary C.3, $\pi=\theta_{2n}^\epsilon(\sigma)$ can not be square-integrable. This contradicts with our assumption. It follows that
\[
	\phi=\phi_\sigma\chi_W^{-1}\chi_V\oplus\chi_V
\]
is also square-integrable.\\

Now it remains to prove that if $\phi$ is square-integrable, then $\pi$ is square-integrable. Indeed the proof follows from the same idea of the proof of the first part. We omit the details here.
\end{proof}

\begin{proposition}
The map $\calL$ preserves temperedness.
\end{proposition}
\begin{proof}
This automatically follows from our construction.
\end{proof}

\section{Preparations for the proof of Local intertwining relation}
Our next step is to prove that $\calL$ and $\calJ_{\scrW}$ satisfy the LIR. In this section, we first briefly recall the definition of normalized intertwining operators, following \cite{MR3573972} Section 7; and then recall a result in \cite{MR3573972}, which is the ingredient of our later proof. Fix $\varepsilon=\pm1$. In this section, we let $V$ and $W$ be an $\varepsilon$-Hermitian space and an $(-\varepsilon)$-Hermitian space respectively. Put
\begin{equation*}
m=\dim V\quad\textit{and}\quad n=\dim W.
\end{equation*}

\subsection{Parabolic subgroups}
Let $r$ be the Witt index of $V$ and $V_{an}$ an anisotropic kernel of $V$. Choose a basis $\{v_i,v_i^*~|~i=1,\cdots,r\}$ of the orthogonal complement of $V_{an}$ such that
\begin{equation*}
\langle v_i,v_j\rangle_V=\langle v_i^*,v_j^*\rangle_V=0,\quad\langle v_i,v_j^*\rangle_V=\delta_{i,j}
\end{equation*}
for $1\leq i,j\leq r$. Let $k$ be a positive integer with $k\leq r$ and set
\begin{equation*}
X=Ev_1\oplus\cdots\oplus Ev_k,\quad X^*=Ev_1^*\oplus\cdots\oplus Ev_k^*.
\end{equation*}
Let $V_0$ be the orthogonal complement of $X\oplus X^*$ in $V$, so that $V_0$ is a $\varepsilon$-Hermitian space of dimension $m_0=m-2k$ over $E$. We shall write an element in the unitary group $U(V)$ as a block matrix relative to the decomposition $V=X\oplus V_0\oplus X^*$. Let $P=M_P U_P$ be the maximal parabolic subgroup of $U(V)$ stabilizing $X$, where $M_P$ is the Levi component of $P$ stabilizing $X^*$ and $U_P$ is the unipotent radical of $P$. We have 
\begin{equation*}
\begin{aligned}
M_P&=\{m_P(a)\cdot h_0~|~a\in GL(X),h_0\in U(V_0)\},\\
U_P&=\{u_P(b)\cdot u_P(c)~|~b\in\Hom(V_0,X),c\in\Herm(X^*,X)\},
\end{aligned}
\end{equation*}

where
\begin{align*}
m_{P}(a)&=\left(\begin{array}{ccc}
                     a & ~ & ~ \\
                     ~ & 1_{V_{0}} & ~ \\ 
                     ~ & ~ & \left(a^{*}\right)^{-1}
                \end{array}\right),\\
u_{P}(b)&=\left(\begin{array}{ccc}
                     1_{X} & b & -\frac{1}{2} b b^{*}\\ 
                     ~ & 1_{V_{0}} & -b^{*} \\ 
                     ~ & ~ & 1_{X^{*}}
                \end{array}\right),\\
u_{P}(c)&=\left(\begin{array}{ccc}
                     {1_{X}} & {} & {c} \\ 
                     {} & {1_{V_{0}}} & {} \\ 
                     {} & {} & {1_{X^{*}}}
                \end{array}\right),
\end{align*}
and 
\begin{equation*}
\Herm(X^*,X)=\{c\in\Hom(X^*,X)~|~c^*=-c\}.
\end{equation*}
Here, the elements $a^*\in GL(X^*)$, $b^*\in\Hom(X^*,V_0)$, and $c^*\in\Hom(X^*,X)$ are the adjoints of $a$, $b$, and $c$ respectively. In particular, $M_P\simeq GL(X)\times U(V_0)$ and we have a exact sequence
\begin{equation*}
1\lra\Herm(X^*,X)\lra U_P \lra\Hom(V_0,X)\lra1.
\end{equation*}
Put
\begin{equation*}
\rho_P=\frac{m_0+k}{2},\quad w_P=\left(\begin{array}{ccc}
                                    {} & {} & -I_X\\
                                    {} & 1_{V_0} & {}\\
                                    -\varepsilon I_X^{-1} & {} & {}

                                 \end{array}\right),
\end{equation*}
where $I_X\in\Isom(X^*,X)$ is defined by $I_Xv_i^*=v_i$ for $1\leq i\leq k$.\\

Similarly, let $r'$ be the Witt index of $W$ and choose a basis $\{w_i,w_i^*~|~i=1,\cdots,r'\}$ of the orthogonal complement of an anisotropic kernel of $W$ such that
\begin{equation*}
\langle w_i,w_j\rangle_W=\langle w_i^*,w_j^*\rangle_W=0,\quad\langle w_i,w_j^*\rangle_W=\delta_{i,j}
\end{equation*}
for $1\leq i,j\leq r'$. We assume that $k\leq r'$ and set
\begin{equation*}
Y=Ew_1\oplus\cdots\oplus Ew_k,\quad Y^*=Ew_1^*\oplus\cdots\oplus Ew_k^*.
\end{equation*}
Let $W_0$ be the orthogonal complement of $Y\oplus Y^*$ in $W$, so that $W_0$ is a $(-\varepsilon)$-Hermitian space of dimension $n_0=n-2k$ over $E$. Let $Q=M_Q U_Q$ be the maximal parabolic subgroup of $U(W)$ stabilizing $Y$, where $M_Q$ is the Levi component of $Q$ stabilizing $Y^*$ and $U_Q$ is the unipotent radical of $Q$. For $a\in GL(Y)$, $b\in\Hom(W_0,Y)$ and $c\in\Herm(Y^*,Y)$, we define elements $m_Q(a)\in M_Q$ and $u_Q(b),u_Q(c)\in U_Q$ as above. We have $M_Q\simeq GL(Y)\times U(W_0)$ and
\begin{equation*}
1\lra\Herm(Y^*,Y)\lra U_Q \lra\Hom(W_0,Y)\lra 1.
\end{equation*}
Put
\begin{equation*}
\rho_Q=\frac{n_0+k}{2},\quad w_Q=\left(\begin{array}{ccc}
                                    {} & {} & -I_Y\\
                                    {} & 1_{W_0} & {}\\
                                    \varepsilon I_Y^{-1} & {} & {}

                                 \end{array}\right),
\end{equation*}
where $I_Y\in\Isom(Y^*,Y)$ is defined by $I_Yw_i^*=w_i$ for $1\leq i\leq k$.


\subsection{Intertwining operators}\label{IO}
To define the local intertwining operators, firstly we need to choose Haar measures on various groups. For this part, readers may refer to \cite{MR3573972}, Section 7.2. We follow their conventions on Haar measures. 
\\

Let $\tau$ be an irreducible (unitary) square-integrable representation of $GL(X)$ on a space $\scrV_\tau$ with central character $\omega_\tau$. For any $s\in\CC$, we realize the representation $\tau_s\coloneqq \tau\otimes|\det|^s$ on $\scrV_\tau$ by setting $\tau_s(a)v\coloneqq|\det a|^s\tau(a)v$ for $a\in GL(X)$ and $v\in\scrV_\tau$. Let $\sigma_0$ be an irreducible tempered representation of $U(V_0)$ on a space $\scrV_{\sigma_0}$. We consider the induced representation
\begin{equation*}
\Ind_P^{U(V)}(\tau_s\boxtimes\sigma_0)
\end{equation*}
of $U(V)$, which is realized on the space of smooth functions $\varPhi_s:U(V)\ra\scrV_\tau\otimes\scrV_{\sigma_0}$ such that
\begin{equation*}
\varPhi_s(um_P(a)h_0h)=|\det a|^{s+\rho_P}\tau(a)\sigma_0(h_0)\varPhi_s(h)
\end{equation*}
for all $u\in U_P$, $a\in GL(X)$, $h_0\in U(V_0)$, and $h\in U(V)$. Let $A_P$ be the split component of the center of $M_P$ and $W(M_P)=N_{U(V)}(A_P)/M_P$ be the relative Weyl group for $M_P$. Noting that $W(M_P)\simeq\ZZ/2\ZZ$, we denote by $w$ the non-trivial element in $W(M_P)$. For any representative $\wt{w}\in U(V)$ of $w$, we define an unnormalized intertwining operator
\begin{equation*}
\calM(\wt{w},\tau_s\boxtimes \sigma_0):\Ind_P^{U(V)}(\tau_s\boxtimes \sigma_0)\lra\Ind_P^{U(V)}(w(\tau_s\boxtimes \sigma_0))
\end{equation*}
by (the meromorphic continuation of) the integral
\begin{equation*}
\calM(\wt{w},\tau_s\boxtimes \sigma_0)\varPhi_s(h)=\int_{U_P}\varPhi_s(\wt{w}^{-1}uh)du,
\end{equation*}
where $w(\tau_s\boxtimes \sigma_0)$ is the representation of $M_P$ on $\scrV_\tau\otimes\scrV_{\sigma_0}$ given by
\begin{equation*}
(w(\tau_s\boxtimes \sigma_0))(m)=(\tau_s\boxtimes \sigma_0)(\wt{w}^{-1}m\wt{w})
\end{equation*}
for $m\in M_P$.\\

Next we shall normalize the intertwining operator $\calM(\wt{w},\tau_s\boxtimes \sigma_0)$, depending on the choice of a Whittaker datum. Having fixed the additive character $\psi$ and the trace zero element $\delta$, we define the sign $\epsilon(V)$ and use the Whittaker datum 
\[
	\begin{cases}
		\scrW_{\psi^{E}}\quad&\textit{if $\varepsilon=+1$, where }\psi^E=\psi(\frac{1}{2}\Tr_{E/F}(\delta\cdot~));\\
		\scrW_{\psi}\quad&\textit{if }\varepsilon=-1.
	\end{cases}
\]
 
Also, we need to choose the following data appropriately:
\begin{itemize}
\item a representative $\wt{w}$;
\item a normalizing factor $r(w,\tau_s\boxtimes\sigma_0)$;
\item an intertwining isomorphism $\calA_w$.
\end{itemize}
For the representative, we take $\wt{w}\in U(V)$ defined by
\begin{equation*}
\wt{w}=w_P\cdot m_P\left((-1)^{m'}\cdot\kappa_V\cdot J\right)\cdot(-1_{V_0})^{k},
\end{equation*}
where $w_P$ is as in the previous subsections, $m'=[\frac{m}{2}]$,
\begin{equation*}
\kappa_V=\begin{cases}
-\delta\quad&\textit{if }m\textit{ is even and $\varepsilon=+1$};\\
1\quad&\textit{if }m\textit{ is even and $\varepsilon=-1$};\\
-1\quad&\textit{if }m\textit{ is odd and $\varepsilon=+1$};\\
-\delta\quad&\textit{if }m\textit{ is odd and $\varepsilon=-1$},\\
\end{cases}
\end{equation*}
and 
\begin{equation*}
J=\left(\begin{array}{cccc}
                                    {} & {} & {} & (-1)^{k-1}\\
                                    {} & {} & \iddots & {}\\
                                    {} & -1 & {} & {}\\
                                    1 & {} & {} & {}

                                 \end{array}\right)\in GL_k(E).
\end{equation*}
Here, we have identified $GL(X)$ with $GL_k(E)$ using the basis ${v_1,\cdots, v_k}$. In \cite{MR3573972} Section 7.3, it is showed that the representative defined above coincides with the representative defined in \cite{MR3338302} when $\epsilon(V)=1$.\\

Next we define the normalizing factor $r(w,\tau_s\boxtimes\sigma_0)$. Let $\lambda(E/F,\psi)$ be the Langlands $\lambda$-factor and put
\begin{equation*}
\lambda(w,\psi)=\begin{cases}
    \lambda(E/F,\psi)^{(k-1)k/2} \quad&\textit{if }m \textit{ is even};\\
    \lambda(E/F,\psi)^{(k+1)k/2} \quad&\textit{if }m \textit{ is odd}.
\end{cases}
\end{equation*}
Let $\phi_\tau$ and $\phi_0$ be the $L$-parameters of $\tau$ and $\sigma_0$ respectively. We set
\begin{equation*}
r(w,\tau_s\boxtimes\sigma_0)=\lambda(w,\psi)\cdot\gamma(s,\phi_\tau\otimes\phi_0^\vee,\psi_E)^{-1}\cdot\gamma(2s,As^{(-1)^m}\circ\phi_\tau,\psi)^{-1},
\end{equation*}
and the normalized intertwining operator
\begin{equation*}
\calR(w,\tau_s\boxtimes\sigma_0)\coloneqq|\kappa_V|^{k\rho_P}\cdot r(w,\tau_s\boxtimes\sigma_0)^{-1}\cdot\calM(\wt w,\tau_s\boxtimes\sigma_0).
\end{equation*}
\begin{lemma}
The normalized intertwining operators satisfy the multiplicative property
\begin{equation*}
\calR(w,w(\tau_s\boxtimes\sigma_0))\circ\calR(w,\tau_s\boxtimes\sigma_0)=1,
\end{equation*}
as well as the adjoint property
\[
	\calR(w,w(\tau_s\boxtimes\sigma_0))^*=\calR(w,\tau_{-\bar{s}}\boxtimes\sigma_0).
\]
In particular, when $s$ is purely imaginary, $\calR(w,\tau_s\boxtimes\sigma_0)$ is unitary. Hence the normalized intertwining operator $\calR(w,\tau_s\boxtimes\sigma_0)$ is holomorphic at $s=0$.
\end{lemma}
\begin{proof}
An easy computation shows that
\begin{equation*}
	\calM(\wt w,\tau_s\boxtimes\sigma_0)=|\kappa_V|^{-k\rho_P}\ell(\wt w)\circ\calM_{\overline{P}|P}(\tau_s\boxtimes\sigma_0),
\end{equation*}
where
\[
	\ell(\wt w):\Ind_{\overline{P}}^{U(V)}\left(\tau_s\boxtimes\sigma_0\right)\lra\Ind_P^{U(V)}\left(w(\tau_s\boxtimes\sigma_0)\right)
\]
is defined by
\[
	\ell(\wt w)\varPsi_s(h)=\varPsi_s(\wt w^{-1}h)
\]
for $\varPsi\in\Ind_{\overline{P}}^{U(V)}\left(\tau_s\boxtimes\sigma_0\right)$. Here the factor $|\kappa_V|^{k\rho_P}$ arises because of our choices of the Haar measures on $U_P$ and $U_P\times U_{\overline{P}}$ in the definition of $\calM(\wt w,\tau_s\boxtimes\sigma_0)$ and $\calM_{\overline{P}|P}(\tau_s\boxtimes\sigma_0)$. Hence
\begin{align*}
\calR(w&,w(\tau_s\boxtimes\sigma_0))\circ\calR(w,\tau_s\boxtimes\sigma_0)\\
=&r(w,\tau_s\boxtimes\sigma_0)^{-1}\cdot r(w,w(\tau_s\boxtimes\sigma_0))^{-1}\cdot\ell(\wt w^2)\circ\calM_{P|\overline{P}}(\tau_s\boxtimes\sigma_0)\circ\calM_{\overline{P}|P}(\tau_s\boxtimes\sigma_0)\\
=&\lambda(w,\psi)^{-2}\cdot\frac{\gamma\left(s,\phi_\tau\otimes\phi_0^\vee,\psi_E\right)\cdot\gamma\left(2s,As^{(-1)^m}\circ\phi_\tau,\psi\right)}{\gamma\left(s,\phi_\tau\otimes\phi_0^\vee,\psi_E\right)\cdot\gamma\left(2s,As^{(-1)^m}\circ\phi_\tau,\psi\right)}\\
&\quad\quad\quad\times\frac{\gamma\left(-s,\phi_\tau^\vee\otimes\phi_0,\psi_E\right)\cdot\gamma\left(-2s,As^{(-1)^m}\circ\phi_\tau^\vee,\psi\right)}{\gamma\left(-s,\phi_\tau^\vee\otimes\phi_0,\psi_E^{-1}\right)\cdot\gamma\left(-2s,As^{(-1)^m}\circ\phi_\tau^\vee,\psi^{-1}\right)}\cdot\ell(\wt w^2)\\
=&\lambda(w,\psi)^{-2}\cdot\det(\phi_\tau^\vee\otimes\phi_0)(-1)\cdot\det(As^{(-1)^m}\circ\phi_\tau^\vee)(-1)\cdot(\tau_s\boxtimes\sigma_0)(\wt w^2)\\
=&\lambda(w,\psi)^{-2}\cdot\omega_\tau(-1)^m\omega_{E/F}(-1)^{(m-1)mk}\cdot\omega_\tau(-1)^k\omega_{E/F}(-1)^{\dim R\phi_\tau}\cdot\omega_\tau\left(\varepsilon\cdot\kappa_V\left(\kappa_V^c\right)^{-1}\cdot(-1)^{k-1}\right),
\end{align*}
where
\[
	R=\begin{cases}
		\bigwedge^2\quad&\textit{if $m$ is even};\\
		Sym^2\quad&\textit{if $m$ is odd}.
	\end{cases}
\]
It follows that
\[
	\calR(w,w(\tau_s\boxtimes\sigma_0))\circ\calR(w,\tau_s\boxtimes\sigma_0)=1
\]
as desired. The adjoint property can be proved exactly the same as \cite{MR3135650} Proposition 2.3.1.
\end{proof}

Finally we define the intertwining isomorphism. Assume that $w(\tau\boxtimes\sigma_0)\simeq\tau\boxtimes\sigma_0$, which is equivalent to $(\tau^c)^\vee\simeq\tau$. We may take the unique isomorphism
\begin{equation*}
\calA_w:\scrV_\tau\otimes\scrV_{\sigma_0}\lra\scrV_\tau\otimes\scrV_{\sigma_0}
\end{equation*}
such that:
\begin{itemize}
\item $\calA_w\circ (w(\tau\boxtimes\sigma_0))(m)=(\tau\boxtimes\sigma_0)(m)\circ\calA_w$ for all $m\in M_P$;
\item $\calA_w=\calA_w'\otimes1_{\scrV_{\sigma_0}}$ with an isomorphism
    \begin{equation*}
        \calA_w':\scrV_\tau\lra\scrV_\tau
    \end{equation*}
    such that $\Lambda\circ\calA_w'=\Lambda$. Here, $\Lambda:\scrV_\tau\ra\CC$ is the unique (up to a scalar) Whittaker functional with respect to the Whittaker datum $(N_k,\psi_{N_k})$, where $N_k$ is the group of unipotent upper triangular matrices in $GL_k(E)$ and $\psi_{N_k}$ is the generic character of $N_k$ given by $\psi_{N_k}(x)=\psi_E(x_{1,2}+\cdots+x_{k-1,k})$.
\end{itemize}

Note that $\calA_w^2=1_{\scrV_\tau\otimes\scrV_{\sigma_0}}$. We define a self-intertwining operator
\begin{equation*}
R(w,\tau\boxtimes\sigma_0):\Ind_P^{U(V)}(\tau\boxtimes \sigma_0)\lra\Ind_P^{U(V)}(\tau\boxtimes \sigma_0)
\end{equation*}
by
\begin{equation*}
R(w,\tau\boxtimes\sigma_0)\varPhi(h)=\calA_w(\calR(w,\tau\boxtimes\sigma_0)\varPhi(h)).
\end{equation*}
for $\varPhi\in\Ind_P^{U(V)}(\tau\boxtimes \sigma_0)$, and $h\in U(V)$. By construction, 
\begin{equation*}
R(w,\tau\boxtimes\sigma_0)^2=1.
\end{equation*}
We shall also use the notation $R(w,\tau\boxtimes\sigma_0,\psi)$ if we want to emphasize the dependence of $R(w,\tau\boxtimes\sigma_0)$ on the additive character $\psi$.
\begin{remark}\label{IO.Change.Whit.Data}
\begin{enumerate}
\item The normalizing factor we defined here is the same as in \cite{MR3573972} Section 7. It is not exactly the same as the normalizing factor defined in \cite{MR3338302} or \cite{kaletha2014endoscopic}; but they have the same analytic behavior near $s=0$. So the final self-intertwining operator $R(w,\tau\boxtimes\sigma_0)$ we defined here coincides with Mok's when $U(V)$ is quasi-split.
\item In the definition of the self-intertwining operator $R(w,\tau\boxtimes\sigma_0)$, if we replace the additive character $\psi$ by $\psi_a$, where $a\in F^\times$, then it follows from an easy computation that
\[
	R(w,\tau\boxtimes\sigma_0,\psi_a)=\begin{cases}
		R(w,\tau\boxtimes\sigma_0,\psi)\cdot\omega_{\tau}(a)\quad&\textit{if $m$ is even};\\
		R(w,\tau\boxtimes\sigma_0,\psi)\quad&\textit{if $m$ is odd}.
	\end{cases}
\]
In particular, the self-intertwining operator $R(w,\tau\boxtimes\sigma_0)$ only depends on the choice of the Whittaker datum.
\end{enumerate}
\end{remark}

Similarly, we can define the normalized intertwining operator for $U(W)$. We put 
\begin{equation*}
\wt{w}=w_Q\cdot m_Q\left((-1)^{n'}\cdot\kappa_W\cdot J\right)\cdot(-1_{W_0})^{k},
\end{equation*}
where $w_Q$ is as in the previous subsection, and $n'=[\frac{n}{2}]$. Let $\pi_0$ be an irreducible tempered representation of $U(W_0)$. We denote the $L$-parameters of $\tau$ and $\pi_0$ by $\phi_\tau$ and $\phi_0$ respectively. We set
\begin{equation*}
r(w,\tau_s\boxtimes\pi_0)=\lambda(w,\psi)\cdot\gamma(s,\phi_\tau\otimes\phi_0^\vee,\psi_E)^{-1}\cdot\gamma(2s,As^{(-1)^n}\circ\phi_\tau,\psi)^{-1},
\end{equation*}
and the normalized intertwining operator
\begin{equation*}
\calR(w,\tau_s\boxtimes\pi_0)\coloneqq|\kappa_W|^{k\rho_Q}\cdot r(w,\tau_s\boxtimes\pi_0)^{-1}\cdot\calM(\wt w,\tau_s\boxtimes\pi_0).
\end{equation*}
Assume that $w(\tau\boxtimes\pi_0)\simeq\tau\boxtimes\pi_0$, we take an isomorphism $\calA_w$ similarly, and define the self-intertwining operator $R(w,\tau\boxtimes\pi_0)$ by 
\[
	R(w,\tau\boxtimes\pi_0)\varPhi(g)=\calA_w\left(\calR(w,\tau_s\boxtimes\pi_0)\varPhi(g)\right)
\]
for $\varPhi\in\Ind_Q^{U(W)}\left(\tau\boxtimes\pi_0\right)$, and $g\in U(W)$. We have
\begin{equation*}
R(w,\tau\boxtimes\pi_0)^2=1.
\end{equation*}

\subsection{An equivariant map}
In \cite{MR3573972} Section 8, Gan-Ichino constructed an equivariant map. We will apply this map to do some computations in later sections. Now we briefly recall some related results.\\

Let $\tau$ be an irreducible square-integrable representation of $GL_k(E)$, $\pi_0$ be an irreducible tempered representation of $U(W_0)$, and $\sigma_0=\theta_{\underline{\psi},V_0,W_0}(\pi_0)$ be the theta lift of $\pi_0$ to $U(V_0)$.
\begin{proposition}\label{equimap}
\begin{enumerate}
\item There is a family of $U(V)\times U(W)$-equivariant maps
\begin{equation*}
\calT_s:\omega\otimes\Ind_P^{U(V)}(\tau_s^c\chi_W^c\boxtimes\sigma_0^\vee)\lra\Ind_Q^{U(W)}(\tau_s\chi_V\boxtimes\pi_0)
\end{equation*}
parametrized by $s\in\CC$. This family of maps $\calT_s$ is holomorphic in $s$.
\item Assume that $m\geq n$. Let $\varPhi\in\Ind_P^{U(V)}(\tau^c\chi_W^c\boxtimes\sigma_0^\vee)$. If $\varPhi\neq0$, then there exists $\varphi\in\scrS$ such that
\begin{equation*}
\calT_0(\varphi\otimes\varPhi)\neq 0.
\end{equation*}
\end{enumerate}
\end{proposition}
\begin{proof}
See \cite{MR3573972} Lemma 8.1 and Lemma 8.3.
\end{proof}

Let $\phi_\tau$, $\phi_0$, and $\phi_0'$ be the $L$-parameters of $\tau$, $\pi_0$, and $\sigma_0$ respectively. We denote by $\wt w'$ and $\wt w$ the representatives of the non-trivial element in $W(M_P)$ and $W(M_Q)$ respectively, as described in the previous subsection.
\begin{proposition}\label{comp.IO}
The diagram
\begin{equation*}
\begin{CD}
\omega\otimes\Ind_P^{U(V)}\left(\tau_s^c\chi_W^c\boxtimes\sigma_0^\vee\right) @>{\calT_s}>> \Ind_Q^{U(W)}\left(\tau_s\chi_V\boxtimes\pi_0\right)\\
@V{1\otimes\calR(\wt{w}',s)}VV @VV{\calR(\wt{w},s)}V\\
\omega\otimes\Ind_P^{U(V)}\left(w'(\tau_s^c\chi_W^c\boxtimes\sigma_0^\vee)\right) @>{\calT_{-s}}>> \Ind_Q^{U(W)}\left(w(\tau_s\chi_V\boxtimes\pi_0)\right)
\end{CD}
\end{equation*}
commutes up to a scalar. Indeed, for $\varphi\in\scrS$ and $\varPhi_s\in\Ind_P^{U(V)}(\tau_s^c\chi_W^c\boxtimes\sigma_0^\vee)$, we have
\begin{equation*}
\calR(\wt{w},\tau_s\chi_V\boxtimes\pi_0)\calT_s(\varphi\otimes\varPhi_s)=\alpha\cdot\beta(s)\cdot\calT_{-s}(\varphi\otimes\calR(\wt{w}',\tau_s^c\chi_W^c\boxtimes\sigma_0^\vee)\varPhi_s),
\end{equation*}
where
\begin{align*}
\alpha=&\left[\gamma_V^{-1}\cdot\gamma_W\cdot\chi_V\left((-1)^{n'}\cdot\varepsilon\cdot\kappa_W^{-1}\right)\cdot\chi_W\left((-1)^{m'-1}\cdot\kappa_V^{-1}\right)\cdot(\chi_V^{-n}\chi_W^m)(\delta)\right]^k\\
&\times\omega_\tau\left((-1)^{m'+n'-1}\cdot\kappa_V^c\kappa_W^{-1}\right)\cdot\lambda(w',\psi)\cdot\lambda(w,\psi)^{-1}
\end{align*}
and
\begin{align*}
\beta(s)=&L\left(s-s_0+\frac{1}{2},\phi_\tau\right)^{-1}\cdot L\left(-s-s_0+\frac{1}{2},(\phi_\tau^c)^\vee\right)\cdot\gamma\left(-s-s_0+\frac{1}{2},(\phi_\tau^c)^\vee,\psi_E\right)\\
&\times|\kappa_V\kappa_W^{-1}|^{-ks}\cdot\gamma\left(s,\phi_\tau^c\otimes\phi_0'\otimes\chi_W^c,\psi_E\right)^{-1}\cdot\gamma\left(s,\phi_\tau\otimes\phi_0^\vee\otimes\chi_V,\psi_E\right).
\end{align*}
\end{proposition}
\begin{proof}
See \cite{MR3573972} Corollary 8.5.
\end{proof}

\section{Local intertwining relation}
In this section, we prove the LLC we constructed for even unitary groups (i.e. $\calL$ and $\calJ_\scrW$) satisfy the LIR. We retain notations in Section \ref{LLC.Construction.J}.\\

Assume that $\phi\in\Parat(2n)$ is a tempered $L$-parameter, such that
\begin{equation*}
\phi=\phi_\tau\oplus\phi_0\oplus(\phi_\tau^c)^\vee,
\end{equation*}
where $\phi_\tau$ is an irreducible tempered representation of $WD_E$ which corresponds to an irreducible (unitary) discrete series representation $\tau$ of $GL_k(E)$, and $\phi_0\in\Parat(2n_0)$, where $n_0=n-k$. So there is a natural embedding $\calS_{\phi_0}\hookrightarrow\calS_\phi$. Let $\pi_0=\pi(\phi_0,\eta_0)$ be an irreducible tempered representation of $U(W_{2n_0}^\epsilon)$. We can write
\begin{equation*}
W_{2n}^\epsilon=Y\oplus W_{2n_0}^\epsilon\oplus Y^*,
\end{equation*}
where $Y$ and $Y^*$ are $k$-dimensional totally isotropic subspaces of $W_{2n}^\epsilon$ such that $Y\oplus Y^*$ is non-degenerate and orthogonal to $W_{2n_0}^\epsilon$. Let $Q$ be the maximal parabolic subgroup of $U(W_{2n}^\epsilon)$ stabilizing $Y$, and $L$ be the Levi component of $Q$ stabilizing $Y^*$, so that
\begin{equation*}
L\simeq GL(Y)\times U(W_{2n_0}^\epsilon).
\end{equation*}
Our goal is to completely analyze the induced representation $\Ind_Q^{U(W_{2n}^\epsilon)}(\tau\boxtimes\pi_0)$.\\

We divide our proof into three part. In the first part, we analyze the $L$-parameter for each irreducible constituent $\pi$ of $\Ind_Q^{U(W_{2n}^\epsilon)}(\tau\boxtimes\pi_0)$; and as a corollary, we get some information on the reducibility of $\Ind_Q^{U(W_{2n}^\epsilon)}(\tau\boxtimes\pi_0)$. In the second part, we analyze the action of the normalized local intertwining operator $R(w,\tau\boxtimes\pi_0)$ on $\Ind_Q^{U(W_{2n}^\epsilon)}(\tau\boxtimes\pi_0)$. In the last part, we relate the character $\eta=\calJ_\scrW(\pi)$ with $\eta_0$.

\subsection{$L$-parameters and reducibilities}
We first prove that
\begin{proposition}\label{LIR.Parameter}
Let $\pi$ be an irreducible subrepresentation of $\Ind_Q^{U(W_{2n}^\epsilon)}(\tau\boxtimes\pi_0)$. Then the $L$-parameter of $\pi$ is $\phi$.
\end{proposition}
\begin{proof}
We pick up $\epsilon'\in\{\pm\}$ appropriately such that $\sigma\coloneqq\theta_{2n+1}^{\epsilon'}(\pi)$ is non-zero. By Lemma \ref{theta.ind}, we have
\begin{equation*}
\sigma\subset\Ind_P^{U(V_{2n+1}^{\epsilon'})}(\tau\chi_V^{-1}\chi_W\boxtimes\sigma_0),
\end{equation*}
where $P$ is a maximal parabolic subgroup of $U(V_{2n+1}^{\epsilon'})$ with Levi component $GL_k(E)\times U(V_{2n_0+1}^{\epsilon'})$, and $\sigma_0\coloneqq\theta_{2n_0+1}^{\epsilon'}(\pi_0)$. By using the LLC for odd unitary groups, it is easy to see that
\begin{equation*}
\phi_\sigma=\phi_\tau\chi_V^{-1}\chi_W\oplus\phi_{\sigma_0}\oplus(\phi_\tau^c)^\vee\chi_V^{-1}\chi_W.
\end{equation*}
On the other hand, by Lemma \ref{Compatible.Parameter}, we have
\begin{align*}
\phi_\sigma&=\phi_\pi\chi_V^{-1}\chi_W\oplus\chi_W,\\
\phi_{\sigma_0}&=\phi_{0}\chi_V^{-1}\chi_W\oplus\chi_W.
\end{align*}
From these equalities, we get $\phi_\pi=\phi$.
\end{proof}
Recall that there is a natural embedding $\calS_{\phi_0}\hookrightarrow\calS_\phi$ of component groups. We identify $\calS_{\phi_0}$ with a subgroup of $\calS_\phi$ via this embedding. 
\begin{corollary}\label{LIR.Reducibility}
The induced representation $\Ind_Q^{U(W_{2n}^\epsilon)}(\tau\boxtimes\pi_0)$ is semi-simple and multiplicity free. Moreover, we have 
\begin{enumerate}
\item If $\calS_{\phi_0}=\calS_\phi$, then $\Ind_Q^{U(W_{2n}^\epsilon)}(\tau\boxtimes\pi_0)$ is irreducible.
\item If $\calS_{\phi_0}$ is a proper subgroup of $\calS_\phi$, then $\Ind_Q^{U(W_{2n}^\epsilon)}(\tau\boxtimes\pi_0)$ is reducible, and has two inequivalent constituents.
\end{enumerate}
\end{corollary}
\begin{proof}
Since $\tau\boxtimes\pi_0$ is an irreducible unitary representation of $L$, the parabolic induction $\Ind_Q^{U(W_{2n}^\epsilon)}(\tau\boxtimes\pi_0)$ is unitary and of finite length, hence semi-simple. Let $\pi$ be an irreducible constituent of $\Ind_Q^{U(W_{2n}^\epsilon)}(\tau\boxtimes\pi_0)$, and
\[
	m_Q(\pi)=\dim\Hom_{U(W_{2n}^\epsilon)}\left(\pi,\Ind_Q^{U(W_{2n}^\epsilon)}(\tau\boxtimes\pi_0)\right).
\]
As in the proof of Proposition \ref{LIR.Parameter}, there exists $\epsilon'\in\{\pm1\}$, such that $\sigma\coloneqq\theta_{2n+1}^{\epsilon'}(\pi)$ is non-zero and 
\[
	\sigma\subset\Ind_P^{U(V_{2n+1}^{\epsilon'})}(\tau\chi_V^{-1}\chi_W\boxtimes\sigma_0),
\]
where $P$ is a maximal parabolic subgroup of $U(V_{2n+1}^{\epsilon'})$ with Levi component $GL_k(E)\times U(V_{2n_0+1}^{\epsilon'})$, and $\sigma_0\coloneqq\theta_{2n_0+1}^{\epsilon'}(\pi_0)$. By the LLC for odd unitary groups, $\Ind_P^{U(V_{2n+1}^{\epsilon'})}(\tau\chi_V^{-1}\chi_W\boxtimes\sigma_0)$ is multiplicity free. It then follows from Lemma \ref{theta.ind} that
\[
	m_Q(\pi)\leq1.
\]
Hence $\Ind_Q^{U(W_{2n}^\epsilon)}(\tau\boxtimes\pi_0)$ is also multiplicity free. We denote by $JH\left(\Ind_Q^{U(W_{2n}^\epsilon)}(\tau\boxtimes\pi_0)\right)$ the set of irreducible constituents of $\Ind_Q^{U(W_{2n}^\epsilon)}(\tau\boxtimes\pi_0)$. Consider the set
\[
	\bigsqcup_{\pi_0}JH\left(\Ind_Q^{U(W_{2n}^\epsilon)}(\tau\boxtimes\pi_0)\right),
\]
where the disjoint union runs over all $\pi_0\in\Pi_{\phi_0}$. By the Howe duality, Lemma \ref{theta.ind}, and Proposition \ref{LIR.Parameter}, this set is indeed a subset of $\Pi_\phi$.\\

Now suppose that $\calS_{\phi_0}=\calS_\phi$. Then we have
\[
	|\Pi_\phi|\geq\left|\bigsqcup_{\pi_0}JH\left(\Ind_Q^{U(W_{2n}^\epsilon)}(\tau\boxtimes\pi_0)\right)\right|\geq|\Pi_{\phi_0}|.
\]
But in this case, it follows from Corollary \ref{Size.Packet} that
\[
	|\Pi_\phi|=|\wh{\calS_\phi}|=|\wh{\calS_{\phi_0}}|=|\Pi_{\phi_0}|.
\]
Therefore we must have
\[
	\left|JH\left(\Ind_Q^{U(W_{2n}^\epsilon)}(\tau\boxtimes\pi_0)\right)\right|=1
\]
for all $\pi_0\in\Pi_{\phi_0}$. In other words, $\Ind_Q^{U(W_{2n}^\epsilon)}(\tau\boxtimes\pi_0)$ is irreducible.\\

Next suppose that $\calS_{\phi_0}$ is a proper subgroup of $\calS_\phi$. In this case, $\calS_{\phi_0}$ is an index two subgroup of $\calS_\phi$. We first show that for all $\pi_0\in\Pi_{\phi_0}$, we have
\[
	\left|JH\left(\Ind_Q^{U(W_{2n}^\epsilon)}(\tau\boxtimes\pi_0)\right)\right|\geq 2.
\]
In other words, $\Ind_Q^{U(W_{2n}^\epsilon)}(\tau\boxtimes\pi_0)$ is reducible. Let 
\begin{align*}
\phi_0^+&\coloneqq\phi_0\chi_V^{-1}\chi_W\oplus\chi_W,\\
\phi^+&\coloneqq\phi\chi_V^{-1}\chi_W\oplus\chi_W.
\end{align*}
Depending on the relative size of $\calS_{\phi_0^+}$ and $\calS_{\phi^+}$, there are two sub-cases:\\

\underline{Sub-case I:} If $\phi_\tau\neq\chi_V$, then $\calS_{\phi_0^+}$ is also a proper subgroup of $\calS_{\phi^+}$. Pick up any $\epsilon'\in\{\pm1\}$ such that $\sigma_0\coloneqq\theta_{2n_0+1}^{\epsilon'}(\pi_0)$ is non-zero. Then $\sigma_0$ has $L$-parameter $\phi_0^+$, and by the LLC for odd unitary groups,  
\[
	\Ind_P^{U(V_{2n+1}^{\epsilon'})}\Big(\left(\tau\chi_V^{-1}\chi_W\right)^c\boxtimes\sigma_0^\vee\Big)\simeq\Big(\Ind_P^{U(V_{2n+1}^{\epsilon'})}(\tau\chi_V^{-1}\chi_W\boxtimes\sigma_0)\Big)^\vee
\]
is reducible. By Proposition \ref{equimap}, there is a $U(V)\times U(W)$-equivariant map
\begin{equation*}
\calT_0:\omega\otimes\Ind_P^{U(V_{2n+1}^{\epsilon'})}\Big(\left(\tau\chi_V^{-1}\chi_W\right)^c\boxtimes\sigma_0^\vee\Big)\lra\Ind_Q^{U(W_{2n}^\epsilon)}(\tau\boxtimes\pi_0),
\end{equation*}
such that for any irreducible constituent $\sigma$ of $\Ind_P^{U(V_{2n+1}^{\epsilon'})}(\tau\chi_V^{-1}\chi_W\boxtimes\sigma_0)$, the restriction $\calT_0\big|_{\omega\otimes\sigma^\vee}$ is non-vanishing, and its image is just $\theta_{2n}^\epsilon(\sigma)$. Hence $\Ind_Q^{U(W_{2n}^\epsilon)}(\tau\boxtimes\pi_0)$ at least contains all these $\theta_{2n}^\epsilon(\sigma)$ as subrepresentations. In particular, $\Ind_Q^{U(W_{2n}^\epsilon)}(\tau\boxtimes\pi_0)$ is also reducible.\\

\underline{Sub-case II:} If $\phi_\tau=\chi_V$, then the natural embedding $\calS_{\phi_0^+}\hookrightarrow\calS_{\phi^+}$ is an isomorphism. Our assumptions in this sub-case imply that $\chi_V\not\subset\phi_0$. By Corollary \ref{Lift.Up}, both $\sigma_0^+\coloneqq\theta_{2n_0+1}^+(\pi_0)$ and $\sigma_0^-\coloneqq\theta_{2n_0+1}^-(\pi_0)$ are non-zero. Moreover, for $\epsilon'\in\{\pm1\}$, $\sigma_0^{\epsilon'}$ has $L$-parameter $\phi_0^+$, and by the LLC for odd unitary groups,
\[\Ind_P^{U(V_{2n+1}^{\epsilon'})}\Big(\left(\tau\chi_V^{-1}\chi_W\right)^c\boxtimes\left(\sigma_0^{\epsilon'}\right)^\vee\Big)\simeq\Big(\Ind_P^{U(V_{2n+1}^{\epsilon'})}\left(\tau\chi_V^{-1}\chi_W\boxtimes\sigma_0^{\epsilon'}\right)\Big)^\vee\]
is irreducible. Similar to Sub-case I, there are non-vanishing $U(V)\times U(W)$-equivariant maps
$$\calT_0^{\epsilon'}:\omega\otimes\Ind_P^{U(V_{2n+1}^{{\epsilon'}})}\Big(\left(\tau\chi_V^{-1}\chi_W\right)^c\boxtimes\left(\sigma_0^{\epsilon'}\right)^\vee\Big)\lra\Ind_Q^{U(W_{2n}^\epsilon)}(\tau\boxtimes\pi_0).$$
Let 
$$\pi^{\epsilon'}\coloneqq\Image(\calT_0^{\epsilon'})\subset\Ind_Q^{U(W_{2n}^\epsilon)}(\tau\boxtimes\pi_0).$$
Then by the Howe duality, we know that $\pi^{\epsilon'}$ is irreducible and is just the theta lift of $\Ind_P^{U(V_{2n+1}^{\epsilon'})}\left(\tau\chi_V^{-1}\chi_W\boxtimes\sigma_0^{\epsilon'}\right)$. Since $\pi^{\epsilon'}$ has $L$-parameter $\phi$ and $\chi_V\subset\phi$, it follows from Lemma \ref{Lift.Down} that  
\begin{equation*}
\pi^+\not\simeq\pi^-,
\end{equation*}
which implies that $\Ind_Q^{U(W_{2n}^\epsilon)}(\tau\boxtimes\pi_0)$ is reducible.\\

Now similar to the previous case, we have
\[
	|\Pi_\phi|\geq\left|\bigsqcup_{\pi_0}JH\left(\Ind_Q^{U(W_{2n}^\epsilon)}(\tau\boxtimes\pi_0)\right)\right|\geq2|\Pi_{\phi_0}|.
\]
Again in this case, Corollary \ref{Size.Packet} forces these inequalities to be equalities. Therefore we conclude that
\[
	\left|JH\left(\Ind_Q^{U(W_{2n}^\epsilon)}(\tau\boxtimes\pi_0)\right)\right|=2
\]
for all $\pi_0\in\Pi_{\phi_0}$. This completes the proof.
\end{proof}

\subsection{Actions of intertwining operators}
In the previous subsection, we showed that $\Ind_Q^{U(W_{2n}^\epsilon)}(\tau\boxtimes\pi_0)$ is semi-simple and multiplicity free. In this subsection, we prove the following: 
\begin{proposition}\label{LIR.Action.LIO}
Assume that $\phi_\tau$ is conjugate self-dual. Let $\pi=\pi(\phi,\eta)$ be an irreducible constituent of $\Ind_Q^{U(W_{2n}^\epsilon)}(\tau\boxtimes\pi_0)$. Then the restriction of the normalized intertwining operator $R(w,\tau\boxtimes\pi_0)$ to $\pi$ is the scalar multiplication by
\begin{equation*}
\begin{cases}
\epsilon^k\cdot\eta(a_\tau)\quad &\textit{if } \phi_\tau\textit{ is conjugate symplectic};\\
\epsilon^k &\textit{if } \phi_\tau\textit{ is conjugate orthogonal},
\end{cases}
\end{equation*}
where $a_\tau$ is the element in $\calS_\phi$ corresponding to $\phi_\tau$.
\end{proposition}
\begin{proof}
Since $\Ind_Q^{U(W_{2n}^\epsilon)}(\tau\boxtimes\pi_0)$ is multiplicity free, the restriction of $R\left(w,\tau\boxtimes\pi_0\right)$ to $\pi$ gives a self intertwining operator of $\pi$. Hence by Schur's Lemma, $R\left(w,\tau\boxtimes\pi_0\right)$ acts on $\pi$ by a scalar. Let's denote this scalar by $\scrR(\pi)$. We want to relate the scalar $\scrR(\pi)$ with the character $\eta$.\\

Let
\begin{equation*}
\epsilon'=\begin{cases}
+\quad&\textit{if } \theta_{2n+1}^+(\pi)\neq 0;\\
-\quad&\textit{otherwise},
\end{cases}
\end{equation*}
and let $\sigma\coloneqq\theta_{2n+1}^{\epsilon'}(\pi)$ (which is non-zero by the conservation relation). Recall that there is a natural embedding of component groups $\calS_\phi\hookrightarrow\calS_{\phi_\sigma}$, and it follows from our construction that $\eta=\eta_\sigma\big|_{\calS_\phi}$. According to Lemma \ref{theta.ind}, we have
\begin{equation*}
\sigma\subset\Ind_P^{U(V_{2n+1}^{\epsilon'})}(\tau\chi_V^{-1}\chi_W\boxtimes\sigma_0),
\end{equation*}
where $P$ is a maximal parabolic subgroup of $U(V_{2n+1}^{\epsilon'})$ with Levi component $GL_k(E)\times U(V_{2n_0+1}^{\epsilon'})$, and $\sigma_0\coloneqq\theta_{2n_0+1}^{\epsilon'}(\pi_0)$. Hence
\begin{equation*}
\sigma^\vee\subset\Ind_P^{U(V_{2n+1}^{\epsilon'})}\Big(\left(\tau\chi_V^{-1}\chi_W\right)^c\boxtimes\sigma_0^\vee\Big).
\end{equation*}
By Lemma \ref{equimap}, there exists a $U(V_{2n+1}^{\epsilon'})\times U(W_{2n}^\epsilon)$-equivariant map
\begin{equation*}
\calT_0:\omega\otimes\Ind_P^{U(V_{2n+1}^{\epsilon'})}\Big(\left(\tau\chi_V^{-1}\chi_W\right)^c\boxtimes\sigma_0^\vee\Big) \lra \Ind_Q^{U(W_{2n}^\epsilon)}\left(\tau\boxtimes\pi_0\right),
\end{equation*}
whose restriction to $\omega\otimes \sigma^\vee$ gives an epimorphism
\begin{equation*}
\calT_0:\omega\otimes\sigma^\vee\lra\pi.
\end{equation*}
Applying Proposition \ref{comp.IO}, we get
\begin{equation*}
\scrR(\pi)=\alpha\cdot\beta(0)\cdot\scrR(\sigma^\vee),
\end{equation*}
where $\scrR(\sigma^\vee)$ is the scalar defined by the action of the normalized intertwining operator $R\left(w',(\tau\chi_V^{-1}\chi_W)^c\boxtimes\sigma_0^\vee\right)$ on $\sigma^\vee$. Following the calculation in \cite{MR3573972} Section 8.4, we have
\begin{equation*}
\epsilon^k\cdot(\epsilon')^k\cdot\alpha\cdot\beta(0)=1.
\end{equation*}
Then one can easily deduce the desired formula for $\scrR(\pi)$ from these two equalities and the LLC for odd unitary groups.
\end{proof}

\subsection{Matching characters of component groups}
Let $\pi$ be an irreducible constituent of $\Ind_Q^{U(W_{2n}^\epsilon)}(\tau\boxtimes\pi_0)$. We showed in Proposition \ref{LIR.Parameter} that the $L$-parameter of $\pi$ is $\phi$. In this subsection, we are going to relate the character $\eta=\calJ_\scrW(\pi)$ of $\calS_\phi$ with $\eta_0$.\\

We first consider a special case.
\begin{lemma}\label{LIR.Compatibility.1st}
Assume that the natural embedding $\calS_{\phi_0}\hookrightarrow\calS_\phi$ is an isomorphism. Then
\[
	\eta\big|_{\calS_{\phi_0}}=\eta_0.
\]
\end{lemma}
\begin{proof}
Similar to the proof of Proposition \ref{LIR.Action.LIO}, we can pick up $\epsilon'_0,\epsilon'\in\{\pm\}$ appropriately such that $\sigma_0\coloneqq\theta_{2n_0+1}^{\epsilon'_0}(\pi_0)$ and $\sigma\coloneqq\theta_{2n+1}^{\epsilon'}(\pi)$ are non-zero, and by our construction,
\begin{align*}
	\eta_0&=\eta_{\sigma_0}\big|_{\calS_{\phi_0}},\\
	\eta&=\eta_{\sigma}\big|_{\calS_{\phi}}.
\end{align*}
One can easily check case-by-case that under our assumption,
\begin{equation*}\tag{\dag}\label{Eqn.dagger}
	\epsilon'_0=\epsilon'.
\end{equation*}
On the other hand, by Lemma \ref{theta.ind}, we have
\begin{equation*}
\sigma\subset\Ind_P^{U(V_{2n+1}^{\epsilon'})}(\tau\chi_V^{-1}\chi_W\boxtimes\sigma_0),
\end{equation*}
where $P$ is a maximal parabolic subgroup of $U(V_{2n+1}^{\epsilon'})$ with Levi component $GL_k(E)\times U(V_{2n_0+1}^{\epsilon'})$. We have a commutative diagram
\begin{equation*}
\begin{CD}
\calS_{\phi_0} @>>> \calS_\phi\\
@VVV @VVV\\
\calS_{\phi_{\sigma_0}} @>>> \calS_{\phi_\sigma}
\end{CD}
\end{equation*}
Here every arrow in this diagram is the natural one. Hence we get
\begin{align*}
\eta\big|_{\calS_{\phi_0}}&=\left(\eta_{\sigma}\big|_{\calS_\phi}\right)\Big|_{\calS_{\phi_0}}\quad &\textit{(by our construction of }\eta\textit{)}\\
&=\left(\eta_\sigma\big|_{\calS_{\phi_{\sigma_0}}}\right)\Big|_{\calS_{\phi_0}} &\textit{(by the commutative diagram)}\\
&=\eta_{\sigma_0}\big|_{\calS_{\phi_0}} &\textit{(by the LLC for odd unitary groups)}\\
&=\eta_0. &\textit{(by our construction of }\eta_0\textit{)}
\end{align*}
\end{proof} 
Here in this lemma, the assumption is only used to guarantee that the equality (\ref{Eqn.dagger}) holds, which may fail in the general case. With this special case at hand, we can show that
\begin{corollary}\label{LIR.Compatibility.2nd}
Let $\epsilon'\in\{\pm1\}$. Assume that $\sigma_0\coloneqq\theta_{2n_0+1}^{\epsilon'}(\pi_0)$ is non-zero. Then
\[
	\eta_0=\eta_{\sigma_0}\big|_{\calS_{\phi_0}}.
\]
Here we use the natural embedding $\calS_{\phi_0}\hookrightarrow\calS_{\phi_{\sigma_0}}$ to identify $\calS_{\phi_0}$ with a subgroup of $\calS_{\phi_{\sigma_0}}$. 
\end{corollary}
\begin{proof}
We use an argument similar to that of \cite{MR3788848} Section 7.3. Let $\phi_\rho$ be any irreducible conjugate symplectic subrepresentation of $\phi_0$, which correponds to a square-integrable representation $\rho$ of $GL_d(E)$, for some $d\leq 2n_0$. We can write
\begin{equation*}
{W}_{2n_0+2d}^\epsilon=Y_\rho\oplus W_{2n_0}^\epsilon\oplus Y_\rho^*,
\end{equation*}
where $Y_\rho$ and $Y_\rho^*$ are $d$-dimensional totally isotropic subspaces of ${W}_{2n_0+2d}^\epsilon$ such that $Y_\rho\oplus Y_\rho^*$ is non-degenerate and orthogonal to $W_{2n_0}^\epsilon$. Let $\wt Q$ be the maximal parabolic subgroup of $U({W}_{2n_0+2d}^\epsilon)$ stabilizing $Y_\rho$ and $\wt L$ be its Levi component stabilizing $Y_\rho^*$, so that
\begin{equation*}
\wt L\simeq GL(Y_\rho)\times U({W}_{2n_0}^\epsilon).
\end{equation*}
We consider the induced representation $\wt{\pi}_0\coloneqq\Ind_{\wt Q}^{U({W}_{2n_0+2d}^\epsilon)}(\rho\boxtimes\pi_0)$. By Corollary \ref{LIR.Reducibility}, $\wt{\pi}_0$ is irreducible. Moreover, it follows from Proposition \ref{LIR.Parameter} and Lemma \ref{LIR.Compatibility.1st} that
\[
	\wt{\pi}_0=\pi(\wt\phi_0,\eta_0)
\]
is the element in $\Pi_{\wt\phi_0}$ corresponding to $\eta_0$, where
\[
	\wt\phi_0=\phi_\rho\oplus\phi_0\oplus(\phi_\rho^c)^\vee,
\]
and we use the natural isomorphism $\calS_{\phi_0}\simeq\calS_{\wt\phi_0}$ to identify $\calS_{\phi_0}$ and $\calS_{\wt\phi_0}$.\\

Similarly, we can write
\begin{equation*}
{V}_{2n_0+2d+1}^{\epsilon'}=X_\rho\oplus V_{2n_0+1}^{\epsilon'}\oplus X_\rho^*,
\end{equation*}
where $X_\rho$ and $X_\rho^*$ are $d$-dimensional totally isotropic subspaces of ${V}_{2n_0+2d+1}^{\epsilon'}$ such that $X_\rho\oplus X_\rho^*$ is non-degenerate and orthogonal to $V_{2n_0+1}^{\epsilon'}$. Let $\wt P$ be the maximal parabolic subgroup of $U({V}_{2n_0+2d+1}^{\epsilon'})$ stabilizing $X_\rho$ and $\wt M$ be its Levi component stabilizing $X_\rho^*$, so that
\begin{equation*}
\wt M\simeq GL(X_\rho)\times U(V_{2n_0+1}^{\epsilon'}).
\end{equation*}
Set $\wt{\sigma}_0\coloneqq\Ind_{\wt P}^{U({V}_{2n_0+2d+1}^{\epsilon'})}(\rho\chi_V^{-1}\chi_W\boxtimes\sigma_0)$. By the LLC for odd unitary groups, $\wt{\sigma}_0$ is irreducible. Moreover, we have
\[
	\wt{\sigma}_0=\pi(\wt\phi_{\sigma_0},\eta_{\sigma_0})
\]
is the element in $\Pi_{\wt\phi_{\sigma_0}}$ corresponding to $\eta_{\sigma_0}$, where
\[
	\wt\phi_{\sigma_0}=\phi_\rho\chi_V^{-1}\chi_W\oplus\phi_{\sigma_0}\oplus\left((\phi_\rho\chi_V^{-1}\chi_W)^c\right)^\vee,
\]
and we use the natural isomorphism $\calS_{\phi_{\sigma_0}}\simeq\calS_{\wt\phi_{\sigma_0}}$ to identify $\calS_{\phi_{\sigma_0}}$ and $\calS_{\wt\phi_{\sigma_0}}$.\\

Recall that by Proposition \ref{equimap}, there exists a non-zero $U({V}_{2n_0+2d+1}^{\epsilon'})\times U({W}_{2n_0+2d}^\epsilon)$-equivariant epimorphism
\[
	\wt\calT_0:\omega\otimes\wt\sigma_0^\vee\lra\wt\pi_0.
\]
Applying Proposition \ref{comp.IO}, we get
\begin{equation*}
\scrR(\wt\pi_0)=\alpha\cdot\beta(0)\cdot\scrR(\wt\sigma_0^\vee),
\end{equation*}
where $\scrR(\wt\pi_0)$ is the scalar defined by the action of the normalized intertwining operator $R\left(w,\rho\boxtimes\pi_0\right)$ on $\wt\pi_0$, and $\scrR(\wt\sigma_0^\vee)$ is defined similarly. Following the calculation in \cite{MR3573972} Section 8.4, we have
\begin{equation*}
\epsilon^k\cdot(\epsilon')^k\cdot\alpha\cdot\beta(0)=1.
\end{equation*}
Combining these two equalitites, the LLC for odd unitary groups, and Proposition \ref{LIR.Action.LIO}, we get
\[
	\eta_0(a_\rho)=\eta_{\sigma_0}(a'_\rho),
\]
where $a_\rho$ is the element in $\calS_{\phi_0}$ corresponding to $\phi_\rho$, and $a'_\rho$ is the element in $\calS_{\phi_{\sigma_0}}$ corresponding to $\phi_\rho\chi_V^{-1}\chi_W$. Since $\phi_\rho$ is arbitrary, we deduce that
\[
	\eta_0=\eta_{\sigma_0}\big|_{\calS_{\phi_0}}.
\]
This completes the proof.
\end{proof}
Finally we can prove the general case:
\begin{proposition}\label{LIR.Compatibility.Fin}
We have
\[
	\eta\big|_{\calS_{\phi_0}}=\eta_0.
\]
\end{proposition}
\begin{proof}
The proof is almost the same as that of Lemma \ref{LIR.Compatibility.1st}. Although the equality (\ref{Eqn.dagger}) may not hold anymore without the assumption in Lemma \ref{LIR.Compatibility.1st}, we can still appeal to Corollary \ref{LIR.Compatibility.2nd} to complete the comparison of $\eta\big|_{\calS_{\phi_0}}$ and $\eta_0$.
\end{proof}
Combining Proposition \ref{LIR.Parameter}, Corollary \ref{LIR.Reducibility}, Proposition \ref{LIR.Action.LIO}, and Proposition \ref{LIR.Compatibility.Fin}, we get
\begin{proposition}\label{LIR.Fin}
The LIR holds for the LLC we constructed for even unitary groups.
\end{proposition}

\section{Completion of the proof}
Now we are equipped with enough powerful arms and able to complete the proof of our main result Theorem \ref{MainTheorem}. In this section, to simplify notations, we let $V^\epsilon$ be the $(2n+1)$-dimensional Hermitian space over $E$ with sign $\epsilon$, and $U(V^\epsilon)$ be the unitary group associated to $V^\epsilon$. Similarly, we let $W^\epsilon$ be the $2n$-dimensional skew-Hermitian space over $E$ with sign $\epsilon$, and $U(W^\epsilon)$ be the unitary group associated to $W^\epsilon$. The idea of many proofs in this section comes from \cite{MR3788848} Section 7.3.
\subsection{Comparison with LLC \`a la Mok}
In this subsection, we compare the LLC for even unitary groups constructed in Section \ref{LLC.Construction} with the LLC for quasi-split unitary groups constructed by Mok in \cite{MR3338302}.\\

Fix a Whittaker datum $\scrW$ of $U(W^+)$. Let $\pi$ be an irreducible smooth representation of $U(W^+)$. Recall that in Section \ref{LLC.Construction}, we associted a pair
\[
	(\phi=\calL(\pi),\eta=\calJ_\scrW(\pi))
\]
to $\pi$. Also, in \cite{MR3338302}, Mok associated a pair
\[
	(\phi^M=\calL^+(\pi),\eta^M=\calJ_\scrW^+(\pi))
\]
to $\pi$. Moreover, the LLC for quasi-split unitary groups constructed by Mok satisfies all properties listed in Theorem \ref{MainTheorem}. 
\begin{theorem}
We have
\[
	\phi=\phi^M\quad\textit{and}\quad\eta=\eta^M.
\]
\end{theorem}
\begin{proof}
Since both two LLC are compatible with Langlands quotients, without loss of generality, we may assume that $\pi$ is tempered. Then the desired conclusion follows from Proposition \ref{WeakPrasadConj.QS} and the same argument as that of Corollary \ref{LIR.Compatibility.2nd}.
\end{proof}
\begin{remark}
Similarly, one can easily show that the bijection $\calJ_\scrW$ constructed in Section \ref{LLC.Construction.J} is independent of the choice of the datum $\underline{\psi}=(\psi,\chi_V,\chi_W,\delta)$, but only depends on the choice of the Whittaker datum $\scrW$.
\end{remark}
As a consequence of this comparison, we deduce
\begin{proposition}
The LLC we constructed for even unitary groups satisfies following properties:
\begin{enumerate}
\item Let $\pi=\pi(\phi,\eta)$ be the element in $\Pi_\phi$ corresponding to $\eta$. Then $\pi$ is a representation of $U(W^\epsilon)$ if and only if $\eta(z_\phi)=\epsilon$.
\item Assume that $\phi$ is a tempered $L$-parameter, then there is an unique $\scrW$-generic representation of $U(W^+)$ in $\Pi_\phi$ corresponding to the trivial character of $\calS_\phi$.
\end{enumerate}
\end{proposition}

\subsection{Twisting by characters}
In this subsection, we prove a formula which concerns the behavior of the LLC we constructed with respect to twisting by characters.\\

Let $\pi=\pi(\phi,\eta)$ be the representation of $U(W^\epsilon)$ in $\Pi_\phi$ corresponding to $\eta$, where $\epsilon=\eta(z_\phi)$. Let $\chi$ be a character of $E^1$, and let $\widetilde{\chi}$ to be the base change of $\chi$, i.e. the pull-back of $\chi$ along
\begin{align*}
E^\times&\rightarrow E^1\\
x&\mapsto x/c(x).
\end{align*}
Let $\pi\chi\coloneqq\pi\otimes(\chi\circ\det)$. Denote by $\phi_\chi$ the $L$-parameter of $\pi\chi$.
\begin{lemma}\label{Twist.Character.Parameter}
We have $\phi_\chi=\phi\cdot\widetilde{\chi}$.
\end{lemma}
\begin{proof}
We first assume that $\pi$ is square-integrable. Then $\pi\chi$ is also square-integrable. By Proposition \ref{Preserve.S-I}, we can write
\[
	\phi=\sum_i\phi_i	
\]
with pairwise inequivalent irreducible conjugate symplectic representation $\phi_i$ of $WD_E$. For each $i$, we may regard $\phi_i$ as an $L$-parameter of $GL_{k_i}(E)$, where $k_i=\dim \phi_i$. We denote by $\rho_i$ the irreducible square-integrable representation of $GL_{k_i}(E)$ corresponding to $\phi_i$. Let $\widetilde{W}_{\phi_i}=W^\epsilon\oplus\calH^{k_i}$, where $\calH$ is the (skew-Hermitian) hyperbolic plane. We can decompose $\widetilde{W}_{\phi_i}$ as follows
\begin{equation*}
\widetilde{W}_{\phi_i}=Y_{\phi_i}\oplus W^\epsilon\oplus Y_{\phi_i}^*,
\end{equation*}
where $Y_{\phi_i}$ and $Y_{\phi_i}^*$ are $k_i$-dimensional totally isotropic subspaces of $\widetilde{W}_{\phi_i}$ such that $Y_{\phi_i}\oplus Y_{\phi_i}^*\simeq\calH^{k_i}$ and orthogonal to $W^\epsilon$. Let $\wt Q_{\phi_i}$ be the maximal parabolic subgroup of $U(\wt{W}_{\phi_i})$ stabilizing $Y_{\phi_i}$ and $\wt L_{\phi_i}$ be its Levi component stabilizing $Y_{\phi_i}^*$, so that
\begin{equation*}
\wt L_{\phi_i}\simeq GL(Y_{\phi_i})\times U(W^\epsilon).
\end{equation*}
Consider the induced representation
\[
	\Ind_{\wt Q_{\phi_i}}^{U(\wt{W}_{\phi_i})}\left(\rho_i\wt\chi\boxtimes\pi\chi\right)\simeq\Ind_{\wt Q_{\phi_i}}^{U(\wt{W}_{\phi_i})}\left(\rho_i\boxtimes\pi\right)\otimes(\chi\circ\det).
\]
By Proposition \ref{LIR.Fin}, $\Ind_{\wt Q_{\phi_i}}^{U(\wt{W}_{\phi_i})}\left(\rho_i\boxtimes\pi\right)$ is irreducible. Hence the induced representation $\Ind_{\wt Q_{\phi_i}}^{U(\wt{W}_{\phi_i})}\left(\rho_i\wt\chi\boxtimes\pi\chi\right)$ is also irreducible. Again by Proposition \ref{LIR.Fin}, it follows that
\[
	\phi_i\cdot\wt\chi\subset\phi_\chi.
\]
This containment holds for all $i$. Therefore we must have
\[
	\phi_\chi=\sum\phi_i\cdot\wt\chi=\phi\cdot\wt\chi.
\]

When $\pi$ is tempered but not square-integrable, the lemma follows from Proposition \ref{LIR.Fin} and induction in stages. In the general case, the lemma follows from the compatibility of the LLC with Langlands quotients.
\end{proof}

Next we consider the character $\eta_{\pi\chi}$ of $\calS_{\phi_\chi}$ associated to $\pi\chi$.
\begin{lemma}
If we use the natural isomorphism $\calS_\phi\simeq\calS_{\phi_\chi}$ to identify them, then we have 
\[
	\eta_{\pi\chi}=\eta.
\]
\end{lemma}
\begin{proof}
Since the LLC we constructed for even unitary groups is compatible with Langlands quotients, without loss of generality, we may assume that $\pi$ is tempered.\\

Similar to the proof of Lemma \ref{Twist.Character.Parameter}, given any irreducible conjugate symplectic representation $\phi_i$ of $WD_E$, we define $\widetilde{W}_{\phi_i}$, $\wt Q_{\phi_i}$ and $\wt L_{\phi_i}$. Consider two induced representations
\[
	\Ind_{\wt Q_{\phi_i}}^{U(\wt{W}_{\phi_i})}\left(\rho_i\boxtimes\pi\right)\quad \textit{and} \quad\Ind_{\wt Q_{\phi_i}}^{U(\wt{W}_{\phi_i})}\left(\rho_i\wt\chi\boxtimes\pi\chi\right).
\]
Then the LIR asserts that it is sufficient to prove the equality
\[
	R(w,\rho_i\boxtimes\pi)=R(w,\rho_i\wt\chi\boxtimes\pi\chi),
\]
where $R(w,\rho_i\boxtimes\pi)$ and $R(w,\rho_i\wt\chi\boxtimes\pi\chi)$ are intertwining operators defined in Section \ref{IO}. Consider the following commutative diagram
\begin{equation*}
\begin{CD}
\Ind_{\wt Q_{\phi_i}}^{U(\wt{W}_{\phi_i})}\left(\rho_i\boxtimes\pi\right)\otimes(\chi\circ\det) @>{\calF}>> \Ind_{\wt Q_{\phi_i}}^{U(\wt{W}_{\phi_i})}\left(\rho_i\wt\chi\boxtimes\pi\chi\right)\\
@V{R(w,\rho_i\boxtimes\pi)\otimes1}VV @VV{R(w,\rho_i\wt\chi\boxtimes\pi\chi)}V\\
\Ind_{\wt Q_{\phi_i}}^{U(\wt{W}_{\phi_i})}\left(\rho_i\boxtimes\pi\right)\otimes(\chi\circ\det) @>{\calF}>> \Ind_{\wt Q_{\phi_i}}^{U(\wt{W}_{\phi_i})}\left(\rho_i\wt\chi\boxtimes\pi\chi\right)
\end{CD}
\end{equation*}
where the horizontal arrow
\begin{align*}
\calF:\Ind_{\wt Q_{\phi_i}}^{U(\wt{W}_{\phi_i})}\left(\rho_i\boxtimes\pi\right)\otimes(\chi\circ\det)&\lra\Ind_{\wt Q_{\phi_i}}^{U(\wt{W}_{\phi_i})}\left(\rho_i\wt\chi\boxtimes\pi\chi\right)
\end{align*}
is given by
\[
	\calF(\varPhi)(g)=\chi(\det g)\varPhi(g)
\]
for $\varPhi\in\Ind_{\wt Q_{\phi_i}}^{U(\wt{W}_{\phi_i})}\left(\rho_i\boxtimes\pi\right)$. Here we realize $\Ind_{\wt Q_{\phi_i}}^{U(\wt{W}_{\phi_i})}\left(\rho_i\boxtimes\pi\right)\otimes(\chi\circ\det)$ on the same space as $\Ind_{\wt Q_{\phi_i}}^{U(\wt{W}_{\phi_i})}\left(\rho_i\boxtimes\pi\right)$, but with the action twisted by $\chi$. The desired equality follows from this commutative diagram easily.
\end{proof}
Combining these two lemmas, we get
\begin{proposition}\label{LLC.Twist.Character}
Let $\pi=\pi(\phi,\eta)$ be the representation of $U(W^\epsilon)$, where $\epsilon=\eta(z_\phi)$. Let $\chi$ be a character of $E^1$. Then
\[
	\pi\chi=\pi\left(\phi\cdot\widetilde{\chi},\eta\right).
\]
Here we use the obvious isomorphism between $\calS_\phi$ and $\calS_{\phi_\chi}$ to identify them.
\end{proposition}

\subsection{Changes of Whittaker data}
In this subsection, we prove a formula which concerns the behavior of the LLC we constructed for even unitary groups with respect to changes of the Whittaker data.\\

Let $\phi\in\Para(2n)$, and $\pi$ be an irreducible smooth representation of $U(W^\epsilon)$ with $L$-parameter $\phi$. Let $\scrW$ and $\scrW'$ be the two Whittaker data of $U(W^+)$. Recall that in Section \ref{LLC.Construction}, we have constructed two bijections
\[
	\calJ_\scrW:\Pi_\phi\lra\wh{\calS_\phi}\quad\textit{and}\quad\calJ_{\scrW'}:\Pi_\phi\lra\wh{\calS_\phi}.
\]
\begin{proposition}\label{Change.Whit.Data}
Let $\eta=\calJ_\scrW(\pi)$ and $\eta'=\calJ_{\scrW'}(\pi)$. Then we have 
\[
	\eta'=\eta\cdot\eta_-,
\]
where $\eta_-$ is a character of $\calS_\phi$ given by
\begin{equation*}
\eta_-(a)=(-1)^{\dim \phi^a}.
\end{equation*}
for $a\in\calS_\phi$.
\end{proposition}
\begin{proof}
As described in Section \ref{whit.data}, we may choose a non-trivial additive character $\psi$ of $F$, such that
\[
	\scrW=\scrW_\psi\quad\textit{and}\quad\scrW'=\scrW_{\psi_{a^w}},
\]
where $a^w\in F^\times\backslash \Nm_{E/F}(E^\times)$. Then by applying an argument similar to that of Corollary \ref{LIR.Compatibility.2nd}, one can see immediately that the desired formula follows from the LIR and Remark \ref{IO.Change.Whit.Data}.
\end{proof}
Using this formula, we are able to prove the last property listed in our main result Theorem \ref{MainTheorem}, which concerns the behavior of the LLC we constructed with respect to taking contragredient.
\begin{proposition}\label{LLC.Contragredient}
Let $\pi=\pi(\phi,\eta)$ be the representation of $U(W^\epsilon)$, where $\epsilon=\eta(z_\phi)$ (with respect to the Whittaker datum $\scrW$). Then 
\[
	\pi^\vee=\pi\left(\phi^\vee,\eta\cdot\nu\right),
\]
where $\nu$ is a character of $\calS_\phi$ given by 
\begin{equation*}
\nu(a)=\omega_{E/F}(-1)^{\dim \phi^a}
\end{equation*}
for $a\in\calS_\phi$. Here we use the obvious isomorphism between $\calS_\phi$ and $\calS_{\phi^\vee}$ to identify them.
\end{proposition}
\begin{remark}
In \cite{MR3194648}, Kaletha proved such a formula using endoscopic character identities for quasi-split groups. Here, based on Kaletha's results for odd unitary groups, we use an elementary arguement to establish the desired formula for all even unitary groups.
\end{remark}
\begin{proof}[Proof of Proposition \ref{LLC.Contragredient}]
Since the LLC we constructed for even unitary groups is compatible with Langlands quotients, without loss of generality, we may assume that $\pi$ is tempered.\\

Pick up a non-trivial additive character $\psi$ of $F$, such that $\scrW=\scrW_\psi$. Let $\underline{\psi}=(\psi,\chi_V,\chi_W,\delta)$ be a tuple of data as described in Section \ref{AuxiliaryDatum}, and $\epsilon'\in\{\pm1\}$ such that
\[
	\Theta_{\underline{\psi},V^{\epsilon'},W^\epsilon}(\pi)=\theta_{\underline{\psi},V^{\epsilon'},W^\epsilon}(\pi)\neq0.
\]
Then we have 
\[
	\theta_{\underline{\psi'},V^{\epsilon'},W^\epsilon}(\pi^\vee\chi_V)\simeq\theta_{\underline{\psi},V^{\epsilon'},W^\epsilon}(\pi)^{MVW}\chi_W,
\]
where $\underline{\psi'}=(\psi^{-1},\chi_V,\chi_W,\delta)$ (see also \cite{MR3166215} Section 6.1). By applying Lemma \ref{Compatible.Parameter}, Proposition \ref{LLC.Twist.Character}, and Theorem \ref{TheoremOdd} to this equality, we get
\[
	\calL(\pi^\vee)=\phi^\vee.
\]
Moreover, by applying Corollary \ref{LIR.Compatibility.2nd} and Theorem \ref{TheoremOdd} to the same equality, we get
\[
	\calJ_{\scrW_{\psi^{-1}}}(\pi^\vee\chi_V)=\calJ_{\scrW_\psi}(\pi).
\]
It then follows from Proposition \ref{LLC.Twist.Character} and Proposition \ref{Change.Whit.Data} that
\[
	\eta_{\pi^\vee}=\eta\cdot\nu
\]
as desired.
\end{proof}
So now, we have finished proving all properties listed in our main result Theorem \ref{MainTheorem}. 
\bibliographystyle{alpha}
\nocite{*}
\bibliography{LLC4UniRef}

\end{document}